\documentclass[onefignum,onetabnum]{siamart171218}

\usepackage{lipsum}
\usepackage{amsfonts}
\usepackage{graphicx}
\usepackage{epstopdf}
\usepackage{algorithm}
\usepackage{algorithmic}
\usepackage{tikz}

\usepackage{multirow}
\usepackage{listings}
\usepackage{mathtools}
\usepackage{latexsym,amsmath,amsfonts,amscd}
\usepackage{subfigure}
\usepackage{bbm}

\newtheorem{thm}{Theorem}[section]
\newtheorem{rmk}{Remark}

\ifpdf
  \DeclareGraphicsExtensions{.eps,.pdf,.png,.jpg}
\else
  \DeclareGraphicsExtensions{.eps}
\fi

\newcommand{\creflastconjunction}{, and~}

\headers{Superconvergence of  approximated coefficients}{H. Li and X. Zhang}

\title{Superconvergence of $C^0-Q^k$ finite element method for elliptic equations with approximated coefficients
}

\author{Hao Li\thanks{Department of Mathematics,
Purdue University,
150 N. University Street,
West Lafayette, IN 47907-2067
  (\email{li2497@purdue.edu}, \email{zhan1966@purdue.edu})}
\and Xiangxiong Zhang\footnotemark[1]}

\usepackage{amsopn}

\begin{document}

\maketitle

\begin{abstract}
We prove that the superconvergence of $C^0$-$Q^k$ finite element method
at the Gauss Lobatto quadrature points still holds if variable coefficients in an elliptic problem are replaced by their piecewise $Q^k$ Lagrange interpolants at the 
 Gauss Lobatto points in each rectangular cell. In particular, a fourth order finite difference type scheme can be constructed using  $C^0$-$Q^2$ finite element method with $Q^2$ approximated coefficients. 

\end{abstract}



\section{Introduction}

\subsection{Motivations}
Consider solving a variable coefficient Poisson equation 
\begin{equation} \label{poisson}
-\nabla\cdot(a\nabla u)=f,\quad a(x,y)>0
\end{equation}
with homogeneous Dirichlet boundary conditions on  a rectangular domain $\Omega$. 
Assume that the coefficient $a(x,y)$ and the solution $u(x,y)$ are sufficiently smooth. 
 Let $\|u\|_{k,p,\Omega}$ be the norm of Sobolev space $W^{k,p}(\Omega)$.
  For $p=2$, let $H^k(\Omega)=W^{k,2}(\Omega)$ and $\|\cdot\|_{k,\Omega}=\|\cdot\|_{k,2,\Omega}$.
The subindex $\Omega$ will be omitted when there is no confusion, e.g., $\|u\|_{0}$ denotes the $L^2(\Omega)$-norm and $\|u\|_1$ denotes the $H^1(\Omega)$-norm.
The  variational form is to find $u\in H_0^1(\Omega)=\{v\in H^1(\Omega): v|_{\partial \Omega}=0\}$ satisfying
\begin{equation} \label{varpro1}
 A(u,v)=(f,v),\quad \forall v\in H_0^1(\Omega),
\end{equation}
 where  
$A(u,v)=\iint_{\Omega} a\nabla u \cdot \nabla v dx dy$, $ (f,v)=\iint_{\Omega}fv dxdy.$ 
Consider a rectangular mesh with mesh size $h$. 
Let $V_0^h\subseteq H^1_0(\Omega)$ be the continuous finite element space consisting of piecewise $Q^k$ polynomials (i.e., tensor product of piecewise polynomials of degree $k$), then the $C^0$-$Q^k$ finite element solution of \eqref{varpro1} is defined as $u_h\in V_0^h$ satisfying 
\begin{equation}\label{numvarpro1}
 A(u_h,v_h)=(f,v_h),\quad \forall v_h\in V_0^h.
 \end{equation}
 
For implementing finite element method \eqref{numvarpro1}, either some quadrature is used or the coefficient $a(x,y)$ is approximated by polynomials for computing $\iint_{\Omega}a u_h v_h\,dx dy$. 
In this paper, we consider the implementation to approximate the smooth coefficient $a(x,y)$  by its $Q^k$ Lagrangian interpolation polynomial in each cell.
For instance, consider $Q^2$ element in two dimensions, tensor product of 3-point Lobatto quadrature form nine uniform points on each cell, see Figure \ref{mesh}.  By point values of $a(x,y)$ at these nine points, we can obtain a $Q^2$ Lagrange interpolation polynomial on each cell. Let $a_I(x,y)$ and $f_I(x,y)$ denote the piecewise $Q^k$ interpolation of $a(x,y)$ and $f(x,y)$ respectively. For a smooth functions $a\geq C> 0$, the interpolation error on each cell $e$ is $\max_{\mathbf x\in e} |a_I(\mathbf x)-a(\mathbf x)|=\mathcal O(h^{k+1})$ thus $a_I>0$ if $h$ is small enough. So if assuming the mesh is fine enough so that $a_I(x,y)\geq C>0,$  
we consider  the following scheme using the approximated coefficients $a_I(x,y)$:
find  $\tilde u_h\in V_0^h$ satisfying
  \begin{equation}\label{numvarpro2}
 A_I(\tilde u_h,v_h): =\iint_{\Omega} a_I \nabla \tilde u \cdot \nabla v dx dy=\langle f,v_h\rangle_h,\quad \forall v_h\in V_0^h,
 \end{equation}
 where $\langle f,v_h\rangle_h$ denotes using tensor product of $(k+1)$-point Gauss Lobatto quadrature for the integral $(f,v_h)$.
 One can also simplify the computation of the right hand side by using $f_I(x,y)$, so we also consider the scheme to find $\tilde{u}_h$ satisfying
 \begin{equation}\label{numvarpro3}
 A_I(\tilde{u}_h,v_h)=(f_I,v_h),\quad \forall v_h\in V_0^h.
 \end{equation}

 \begin{figure}[h]
 \subfigure[A $n_x\times n_y$ finite difference grid]{\includegraphics[scale=0.8]{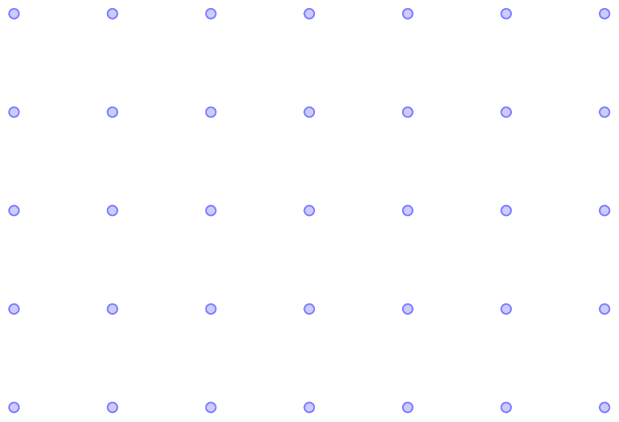} }
 \hspace{.6in}
 \subfigure[The corresponding $(n_x-1)/2\times(n_y-1)/2$ mesh $\Omega_h$ for $Q^2$ element]{\includegraphics[scale=0.8]{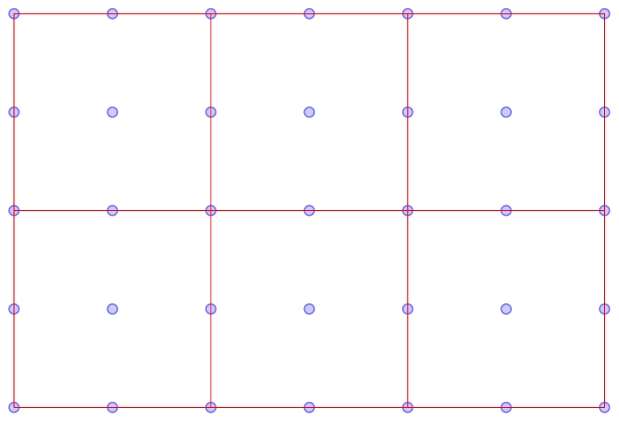}}
\caption{An illustration of meshes. }
\label{mesh}
 \end{figure}
 The schemes \eqref{numvarpro2} and \eqref{numvarpro3} correspond to the equation \begin{equation}
-\nabla\cdot(a_I (x,y)\nabla \tilde u(x,y))=f(x,y).
 \label{Poisson-approximated}
\end{equation}
At first glance, one might expect $(k+1)$-th order accuracy for a numerical method applying to \eqref{Poisson-approximated} due to the interpolation error $a(x,y)-a_I(x,y)=\mathcal O(h^{k+1})$.
But as we will show in Section \ref{sec-main-pdesection}, the difference between exact solutions $u$ and $\tilde u$ to the two elliptic equations \eqref{poisson} and \eqref{Poisson-approximated} is $\mathcal O(h^{k+2})$ in $L^2(\Omega)$-norm under suitable assumptions. 
The main focus of this paper is to  show \eqref{numvarpro2} and \eqref{numvarpro3} are $(k+2)$-th order accurate finite difference type schemes 
via the superconvergence of finite element method.
Such a result is  very interesting from the perspective that 
a fourth order accurate scheme can be constructed even if the coefficients in the equation are approximated by quadratic polynomials, which does not seem to be considered before in the literature.

Since only grid point values of $a(x,y)$ and $f(x,y)$ are needed in  scheme \eqref{numvarpro2} or \eqref{numvarpro3},  they can be regarded as finite difference type schemes. 
Consider a uniform $n_x\times n_y$ grid for a rectangle $\Omega$ with grid points $(x_i,y_j)$ and grid spacing $h$, 
where $n_x$
and $n_y$ are both odd numbers as shown in Figure \ref{mesh}(a). Then   there is a mesh $\Omega_h$ of $(n_x-1)/2\times(n_y-1)/2$ $Q^2$ elements so that Gauss-Lobatto points for all cells in $\Omega_h$ are exactly the finite difference grid points. 
By using the scheme \eqref{numvarpro2} or \eqref{numvarpro3} on the finite element mesh $\Omega_h$ shown in Figure \ref{mesh}(b),
we obtain a fourth order finite difference scheme in the sense that $\tilde u_h$
is fourth order accurate in the discrete 2-norm at all grid points.

In practice the most convenient implementation is to use  tensor product of $(k+1)$-point Gauss Lobatto quadrature for integrals in \eqref{varpro1}, since the standard $L^2(\Omega)$ and $H^1(\Omega)$ error estimates still hold \cite{ciarlet1972combined, ciarlet1991basic} and the Lagrangian $Q^k$ basis are delta functions at these quadrature points. Such a quadrature scheme can be denoted as finding $u_h\in V_0^h$ satisfying
\begin{equation}\label{numvarpro4}
 A_h(u_h,v_h)=\langle f,v_h\rangle_h,\quad \forall v_h\in V_0^h,\end{equation}
where $A_h(u_h,v_h)$  and $\langle f,v_h\rangle_h$ denote using tensor product of $(k+1)$-point Gauss Lobatto quadrature for integrals $A(u_h,v_h)$ and $(f,v_h)$ respectively.
Numerical tests suggest that the approximated coefficient scheme \eqref{numvarpro3} is more accurate and robust than the quadrature scheme \eqref{numvarpro4} in some cases.

\subsection{Superconvergence of $C^0$-$Q^k$ finite element method}
 Standard error estimates of \eqref{numvarpro1} are
 $\|u-u_h\|_{1}\leq C h^{k}\|u\|_{k+1}$ and $\|u-u_h\|_{0}\leq C h^{k+1}\|u\|_{k+1}$ \cite{ciarlet1991basic}.
  At certain quadrature or symmetry points the finite element solution or its derivatives have higher order accuracy, which is called superconvergence.
Douglas and Dupont first proved that continuous finite element method using piecewise polynomial of degree $k$ has $O(h^{2k})$ convergence at the knots  in an one dimensional mesh \cite{douglas1973some, douglas1974galerkin}. In \cite{douglas1974galerkin}, $O(h^{2k})$ was proven to be the best possible convergence rate. 
For $k\geq 2$, $\mathcal O (h^{k+1})$ for the derivatives at Gauss quadrature points and $\mathcal O (h^{k+2})$ for functions values at Gauss-Lobatto quadrature points were proven in \cite{lesaint1979superconvergence, chen1979superconvergent, bakker1982note}.

For two dimensional cases, it was first showed in \cite{douglas1974infty}  that the $(k+2)$-th order superconvergence for $k\geq 2$ at vertices of all rectangular cells in a two dimensional rectangular mesh. 
Namely, the convergence rate at the knots is as least one order higher than the rate globally. 
Later on,  the $2k$-th order (for $k\geq 2$) convergence rate at the knots was proven for $Q^k$ elements solving $-\Delta u=f$, see \cite{chen2013highest, he20172𝑘}.

For the multi-dimensional variable coefficient case, when discussing the superconvergence of derivatives, it can be reduced to the Laplacian case. Superconvergence of tensor product elements for the Laplacian case can be established by extending one-dimensional results \cite{douglas1974infty, wahlbin2006superconvergence}.
See also \cite{He2017} for the superconvergence of the gradient. 
The superconvergence of function values in rectangular elements for the variable coefficient case  were studied  in \cite{chen2001structure} by Chen with M-type projection polynomials and in \cite{lin1996} by  Lin and Yan with the point-line-plane interpolation polynomials. In particular, let $Z_0$ denote the set of tensor product of $(k+1)$-point Gauss-Lobatto quadrature points for all rectangular cells, then the following superconvergence of function values for  $Q^k$ elements  was shown in \cite{chen2001structure}:
 \begin{eqnarray}
\label{super-exactcoef} \left( h^2\sum_{(x,y)\in Z_0}|u(x,y)-u_h(x,y)|^2\right)^{1/2}&\leq& C h^{k+2}\|u\|_{k+2},\quad  k\geq 2,\\
\label{super-exactcoef-max} \max_{(x,y)\in Z_0}|u(x,y)-u_h(x,y)|&\leq &C h^{k+2}|\ln{h}|\|u\|_{k+2,\infty,\Omega},\quad k\geq 2.
\end{eqnarray} 

In general superconvergence of  \eqref{numvarpro1}   has been well studied in the literature. Many superconvergence results are established for interior points away from the boundary for various domains. Our major motivation to study superconvergence is to use it for constructing a finite difference scheme, thus we only consider a rectangular domain for which all Lobatto points can form a finite difference grid.

We are interested in superconvergence of function values for $Q^k$ element when the computation of integrals is simplified. 
 For one-dimensional problems, it was proven in \cite{douglas1974galerkin} that 
$O(h^{2k})$ at knots still holds if  $(k+1)$-point Gauss-Lobatto quadrature is used for $P^2$ element.  Superconvergence of the gradient for using quadrature  was studied in \cite{lesaint1979superconvergence}. 
For multidimensional problems, even though it is possible to show \eqref{super-exactcoef} holds for \eqref{numvarpro1} with accurate enough quadrature,  it is nontrivial to extend the superconvergence proof  to \eqref{numvarpro4} with only $(k+1)$-point Gauss Lobatto quadrature. Superconvergence analysis of the scheme \eqref{numvarpro4} is much more complicated thus will be discussed in another paper \cite{li2019fourth}.

\subsection{Contributions of the paper}

The objective and main motivation of this paper is to construct a fourth order accurate finite difference type scheme  based on the superconvergence of $C^0$-$Q^2$ finite element method using $Q^2$ polynomial coefficients in elliptic equations and demonstrate the accuracy. 
The main result can be easily generalized to higher order cases thus we keep the discussion general to $Q^k$ ($k\geq 2$)
and prove its $(k+2)$-th order superconvergence of function values when
using PDE coefficients are replaced by their $Q^k$ interpolants: \eqref{super-exactcoef} still holds for both schemes \eqref{numvarpro2} and \eqref{numvarpro3}.
 Moreover, \eqref{numvarpro2} and \eqref{numvarpro3} have all finite element method advantages such as the symmetry of the stiffness matrix, which is desired in applications. The scheme \eqref{numvarpro2} or \eqref{numvarpro3} is also an efficient implementation of $C^0$-$Q^k$ finite element method since only $Q^k$  coefficients are needed to retain the $(k+2)$-th order accuracy of function values at the Lobatto points.
 
The paper is organized as follows. In Section \ref{sec-preliminaries}, we introduce the notations and review standard interpolation and quadrature estimates. In Section \ref{sec-projection}, we review the tools to establish superconvergence of function values in $C^0$-$Q^k$ finite element method \eqref{numvarpro1} with a complete proof. 
In Section \ref{sec-main}, we prove the main result on the superconvergence of \eqref{numvarpro2} and \eqref{numvarpro3} in two dimensions with extensions to a general elliptic equation.
All discussion in this paper can be easily extended to the three dimensional case.
Numerical results are given in Section \ref{sec-test}. Section \ref{sec-remark} consists of concluding remarks.

\section{Notations and preliminaries}
\label{sec-preliminaries}
\subsection{Notations}
In addition to the notations mentioned in the introduction, the following notations will be used in the rest of the paper: 
\begin{itemize}
 \item $n$ denotes the dimension of the problem. Even though we discuss everything explicitly for $n=2$, all key discussions  can be
 easily extended to $n=3$. The main purpose of keeping $n$ is for readers to see independence/cancellation  of the dimension $n$ in the proof of some important estimates. 
 \item We only consider a rectangular domain $\Omega$ with its boundary $\partial \Omega$. 
 \item $\Omega_h$ denotes a rectangular mesh with mesh size $h$. Only for convenience, we assume $\Omega_h$ is an uniform mesh and $e=[x_e-h,x_e+h]\times [y_e-h,y_e+h]$ denotes any cell in $\Omega_h$ with 
 cell center $(x_e,y_e)$. {\bf The assumption of an uniform mesh is not essential to the proof}.
  \item $Q^k(e)=\left\{p(x,y)=\sum\limits_{i=0}^k\sum\limits_{j=0}^k p_{ij} x^iy^j, (x,y)\in e\right\}$ is the set of 
 tensor product of polynomials of degree $k$ on a cell $e$.
\item $V^h=\{p(x,y)\in C^0(\Omega_h): p|_e \in Q^k(e),\quad \forall e\in \Omega_h\}$ denotes the continuous piecewise $Q^k$ finite element space on $\Omega_h$.
\item $V^h_0=\{v_h\in V^h: v_h=0 \quad\mbox{on}\quad \partial \Omega \}.$
\item The norm and seminorms for $W^{k,p}(\Omega)$ and $1\leq p<+\infty$, with standard modification for $p=+\infty$:
$$
 \|u\|_{k,p,\Omega}=\left(\sum\limits_{i+j\leq  k}\iint_{\Omega}|\partial_x^i\partial_y^ju(x,y)|^pdxdy\right)^{1/p},
 $$
$$
 |u|_{k,p,\Omega}=\left(\sum\limits_{i+j= k}\iint_{\Omega}|\partial_x^i\partial_y^ju(x,y)|^pdxdy\right)^{1/p},
 $$
$$
 [u]_{k,p,\Omega}=\left(\iint_{\Omega}|\partial_x^k u(x,y)|^pdxdy+\iint_{\Omega}|\partial_y^k u(x,y)|^p dxdy\right)^{1/p}.
 $$
 Notice that $[u]_{k+1,p,\Omega}=0$ if $u$ is a $Q^k$ polynomial.
 \item $\|u\|_{k,\Omega}$, $|u|_{k,\Omega}$ and $[u]_{k,\Omega}$ denote 
 norm and seminorms for $H^k(\Omega)=W^{k,2}(\Omega)$.
 \item When there is no confusion, $\Omega$ may be dropped in the norm and seminorms. 
 \item For any $v_h\in V_h$, $1\leq p<+\infty $ and $k\geq 1$, $$\|v_h\|_{k,p,\Omega}:= \left[\sum_e\|v_h\|_{k,p,e}^p\right]^{\frac1p}, \quad 
 |v_h|_{k,p,\Omega}:= \left[\sum_e|v_h|_{k,p,e}^p\right]^{\frac1p}.$$
\item Let $Z_{0,e}$ denote the set of $(k+1)\times (k+1)$ Gauss-Lobatto points on a cell $e$.
\item $Z_0=\bigcup_e Z_{0,e}$ denotes all Gauss-Lobatto points in the mesh $\Omega_h$.
\item 
Let $\|u\|_{2,Z_0}$ and $\|u\|_{\infty,Z_0}$
denote the discrete 2-norm and the maximum norm over $Z_0$ respectively:
\[\|u\|_{2,Z_0}=\left[h^2\sum_{(x,y)\in Z_0} |u(x,y)|^2\right]^{\frac12},\quad \|u\|_{\infty,Z_0}=\max_{(x,y)\in Z_0} |u(x,y)|.\]
\item For a smooth function $a(x,y)$, let $a_I(x,y)$ denote its piecewise $Q^k$ Lagrange interpolant at $Z_{0,e}$ on each cell $e$, i.e., $a_I\in V^h$ satisfies:
\[a(x,y)=a_I(x,y), \quad \forall (x,y)\in Z_0. \]
\item $P^k(t)$ denotes the polynomial of degree $k$ of variable $t$. 
\item $(f,v)$ denotes the inner product in $L^2(\Omega)$:
\[(f,v)=\iint_{\Omega} fv\, dxdy.\]
\item $\langle f,v\rangle_h$ denotes the approximation to $(f,v)$ by using $(k+1)\times(k+1)$-point Gauss Lobatto
quadrature for integration over each cell $e$. 
\end{itemize}
  
  \bigskip
  
The following are commonly used
  tools and facts:
  \begin{itemize}
   \item $\hat K=[-1,1]\times [-1,1]$ denotes a reference cell.
 \item For $v(x,y)$ defined on $e$, consider $\hat v(s, t)=v(sh+ x_e,t h+ y_e)$  defined on $\hat K$. 
 \item For $n$-dimensional problems, the following scaling argument will be used:
\begin{equation}
h^{k-n/p}|v|_{k,p,e}=|\hat v|_{k,p,\hat K},\quad  h^{k-n/p}[v]_{k,p,e}=[\hat v]_{k,p,\hat K}, \quad 1\leq p\leq \infty.
\label{scaling}
\end{equation}
 \item Sobolev's embedding in two and three dimensions: $H^{2}(\hat K)\hookrightarrow C^0(\hat K)$. 
 \item The embedding implies  
  $$\|\hat  f\|_{0,\infty,\hat K}\leq C \|\hat  f\|_{k,2, \hat K}, \forall \hat f\in H^{k}(\hat K), k\geq 2,$$
 $$\|\hat  f\|_{1,\infty,\hat K}\leq C \|\hat  f\|_{k+1,2, \hat K}, \forall\hat f\in H^{k+1}(\hat K), k\geq 2.$$
   \item Cauchy Schwarz inequalities:
  \[\sum_e \|u\|_{k,e}\|v\|_{k,e}\leq \left(\sum_e \|u\|^2_{k,e}\right)^{\frac12}\left(\sum_e \|v\|^2_{k,e}\right)^{\frac12}, 
  \|u\|_{k,1,e}=\mathcal O(h^{\frac{n}{2}}) \|u\|_{k,2,e}.\]
  \item Poincar\'{e} inequality: let $\bar{\hat f}$ be the average of $\hat f\in H^1(\hat K)$ on $\hat K$, then 
  \[ |\hat f-\bar{\hat f}|_{0,p,\hat K}\leq C |\nabla \hat f|_{0,p,\hat K}, \quad p\geq 1.\]
 \item 
For $k\geq 2$, the $(k+1)\times(k+1)$ Gauss-Lobatto quadrature is exact for 
integration of polynomials of degree $2k-1\geq k+1$ on $\hat K$.
\item Any polynomial in $Q^k(\hat K)$ can be uniquely represented by its point values at $(k+1)\times(k+1)$ Gauss Lobatto points on $\hat K$, and it is straightforward to verify that the discrete $2$-norm $\|p\|_{2, Z_0}$ and $L^2(\Omega)$-norm $\|p\|_{0, \Omega}$ are equivalent for a piecewise $Q^k$ polynomial $p\in V^h$. 
\item  Define the projection operator $\hat{\Pi}_1: \hat u \in L^1(\hat K)\rightarrow \hat \Pi_1\hat u\in Q^1(\hat K)$ by 
 \begin{equation}
\iint_{\hat K} (\hat{\Pi}_1 \hat{u} ) w dxdy= \iint_{\hat K} \hat{u} w dxdy,\forall w\in Q^1(\hat K).
\label{projection1}
 \end{equation}
 Notice that $\hat{\Pi}_1$ is a continuous linear mapping from $L^2(\hat K)$ to $H^1(\hat K)$ (or $H^2(\hat K)$) since  all degree of freedoms of $\hat{\Pi}_1 \hat{u}$ can be represented as a linear combination of $\iint_{\hat K} \hat u(s,t) p(s,t)dsdt$ for $p(s,t)=1,s,t,st$
 and by Cauchy Schwarz inequality $|\iint_{\hat K} \hat u(s,t) p(s,t)dsdt|\leq \|\hat u\|_{0,2,\hat K}\|\hat p\|_{0,2,\hat K}\leq C \|\hat u\|_{0,2,\hat K}$.
  \end{itemize}

\subsection{The Bramble-Hilbert Lemma}
By the abstract Bramble-Hilbert Lemma
in \cite{brezzi1975numerical}, with the result  $\|v\|_{m,p,\Omega}\leq 
C(|v|_{0,p,\Omega}+[v]_{m,p,\Omega})$ for any $v\in W^{m, p}(\Omega)$ \cite{smith1961inequalities, agmon2010lectures}, the Bramble-Hilbert Lemma for $Q^k$ polynomials can be stated as (see Exercise 3.1.1 and Theorem 4.1.3 in \cite{ciarlet1978finite}):
\begin{thm}
\label{bh-lemma}
If a continuous linear mapping  $\Pi: H^{k+1}(\hat K)\rightarrow H^{k+1}(\hat K)$
satisfies $\Pi v=v$ for any $v\in Q^k(\hat K)$, then 
\begin{equation}
\|u-\Pi u\|_{k+1,\hat K}\leq C [u]_{k+1, \hat K}, \quad \forall u\in H^{k+1}(\hat K).
\label{bh1}
\end{equation}
Thus if $l(\cdot)$ is a continuous linear form on the space $H^{k+1}(\hat K)$ satisfying
$l(v)=0,\forall v\in Q^k(\hat K),$
then \[|l(u)|\leq C \|l\|'_{k+1, \hat K} [u]_{k+1,\hat K},\quad \forall u\in H^{k+1}(\hat K),\]
where $ \|l\|'_{k+1, \hat K}$ is the norm in the dual space of $H^{k+1}(\hat K)$.
\end{thm}

\subsection{Interpolation and quadrature errors}
For $Q^k$ element ($k\geq 2$), consider $(k+1)\times(k+1)$ Gauss-Lobatto quadrature, which is exact for integration of $Q^{2k-1}$ polynomials. 

It is straightforward to establish the 
interpolation error:
\begin{thm}
\label{interp-theorem}
For a smooth function $a$,  $|a-a_I|_{0,\infty,\Omega}=\mathcal O (h^{k+1})|a|_{k+1,\infty,\Omega}$.
\end{thm}

Let $s_{j}, t_j$ and $w_j$ $(j=1,\cdots, k+1)$ be the Gauss-Lobatto quadrature points and weight for the interval $[-1,1]$. Notice $\hat f$ coincides with its $Q^k$ interpolant $\hat f_I$ at the quadrature points and the quadrature is exact for integration of $\hat f_I$, the quadrature can be expressed on $\hat K$ as 
\[ \sum_{i=1}^{k+1}\sum_{j=1}^{k+1}\hat f(s_i, t_j)w_i w_j=\iint_{\hat K} \hat f_I(x,y)dxdy,\]
thus the quadrature error is related to interpolation error: 
\[ \iint_{\hat K} \hat f(x,y)dxdy-\sum_{i=1}^{k+1}\sum_{j=1}^{k+1}\hat f(s_i, t_j)w_i w_j
= \iint_{\hat K} \hat f(x,y)dxdy- \iint_{\hat K} \hat f_I(x,y)dxdy.\]

We have the following estimates on the quadrature error:
\begin{thm}
For $n=2$ and a sufficiently smooth function $a(x,y)$, if $k\geq 2$ and  $m$ is an integer satisfying $k\leq  m\leq 2k$, 
we have 
\label{quaderror-theorem}
\[  \iint_e a(x,y)dxdy-  \iint_e a_I(x,y)dxdy=\mathcal O(h^{m+\frac{n}{2}})[a]_{m,e}=\mathcal O(h^{m+n})[a]_{m,\infty,e}.\]

\end{thm}
\begin{proof}
Let $E(a)$ denote the quadrature error for function $a(x,y)$ on $e$. 
Let $\hat E(\hat a)$ denote the quadrature error for the function $\hat a(s,t)=a(sh+x_e, th+y_e)$ 
on the reference cell $\hat K$. Then for any $\hat f\in H^{m}(\hat K)$ ($m\geq k\geq2$), since quadrature are represented by point values,  with the Sobolev's embedding 
we have $$|\hat E(\hat f)|\leq C |\hat f|_{0,\infty,\hat K}\leq C\|\hat f\|_{m,2,\hat K}.$$ Thus $\hat E(\cdot)$ is a continuous linear form on $H^{m}(\hat K)$
and  $\hat E(\hat f)=0$ if $\hat f\in Q^{m-1}(\hat K)$. With \eqref{scaling}, the Bramble-Hilbert lemma
implies
\[|E(a)|=h^n|\hat E(\hat a)|\leq Ch^n[\hat a]_{m, 2, \hat K}=\mathcal O(h^{m+\frac{n}{2}})[a]_{m,2, e}=\mathcal O(h^{m+n})[a]_{m,\infty, e}. \]
\end{proof}

\begin{thm}
\label{rhs-estimate}
 If $k\geq 2$, $(f,v_h)-\langle f,v_h\rangle_h =\mathcal O(h^{k+2}) \|f\|_{k+2} \|v_h\|_2,\quad\forall v_h\in V^h.$
\end{thm}
\begin{proof}
 This result is a special case of Theorem 5 in \cite{ciarlet1972combined}.
 For completeness, we include a proof. 
 Let $\hat E(\cdot)$ denote the quadrature error term on the reference cell $\hat K$. 
Consider the projection \eqref{projection1}. Let $\Pi_1$ denote the same 
projection on $e$. Since $\hat\Pi_1$ leaves $Q^0(\hat K)$ invariant,
by the Bramble-Hilbert lemma on $\hat\Pi_1$, we get $[\hat v_h-\hat \Pi_1 \hat v_h]_{1,\hat K}\leq \|\hat v_h-\hat \Pi_1 \hat v_h\|_{1,\hat K} \leq C [\hat v_h]_{1,\hat K}$ thus $[\hat \Pi_1 \hat v_h]_{1,\hat K}\leq [\hat v_h]_{1,\hat K}+[\hat v_h-\hat\Pi_1 \hat v_h]_{1,\hat K}\leq C [\hat v_h]_{1,\hat K}$.
By setting $w=\hat \Pi_1\hat v_h$ in \eqref{projection1}, we get $|\hat{\Pi}_1 \hat{v}_h|_{0,\hat K}\leq  |\hat v_h|_{0,\hat K}$.
 For $k\geq 2$, repeat the proof of Theorem \ref{quaderror-theorem}, we can get 
\[|\hat E(\hat f\hat \Pi_1 \hat v_h)|\leq C [\hat f\hat \Pi_1 \hat v_h]_{k+2,\hat K}
\leq C ([\hat f]_{k+2,\hat K} | \hat\Pi_1 \hat v_h|_{0,\infty,\hat K}+[\hat f]_{k+1,\hat K} |\hat \Pi_1 \hat v_h|_{1,\infty,\hat K}),\]
where the fact $[\hat\Pi_1 \hat v_h]_{l,\infty,\hat K}=0$ for $l\geq 2$ is used. The equivalence of norms over $Q^1(\hat K)$ implies  
\begin{align*}
 |\hat E(\hat f \hat\Pi_1 \hat v_h)|\leq C([\hat f]_{k+2,\hat K} | \hat\Pi_1 \hat v_h|_{0,\hat K}+[\hat f]_{k+1,\hat K} |\hat \Pi_1 \hat v_h|_{1,\hat K})\\ \leq C([\hat f]_{k+2,\hat K} |  \hat v_h|_{0,\hat K}+[\hat f]_{k+1,\hat K} |\hat v_h|_{1,\hat K}). 
\end{align*}
 Next consider
 the linear form $\hat f\in H^k(\hat K)\rightarrow \hat E(\hat f (\hat v_h-\hat\Pi_1 \hat v_h))$. Due to the embedding $H^{k}(\hat K)\hookrightarrow C^0(\hat K)$, it is continuous with operator norm $\leq C\|\hat v_h-\hat\Pi_1 \hat v_h\|_{0,\hat K}$ since
 \begin{align*}|\hat E(\hat f (\hat v_h-\hat\Pi_1 \hat v_h))|\leq C |\hat f (\hat v_h-\hat\Pi_1 \hat v_h)|_{0,\infty,\hat K}\leq C |\hat f |_{0,\infty,\hat K}|\hat v_h-\hat\Pi_1 \hat v_h|_{0,\infty,\hat K} \\
 \leq C\|\hat f\|_{k,\hat K} \|\hat v_h-\hat\Pi_1 \hat v_h\|_{0,\hat K}.\end{align*}
For any 
$\hat f\in Q^{k-1}(\hat K)$, $\hat E(\hat f \hat v_h)=0$.
By the Bramble-Hilbert lemma, we get  
$$|\hat E(\hat f (\hat v_h-\hat\Pi_1 \hat v_h))|\leq C [\hat f]_{k,\hat K} \|\hat v_h-\hat\Pi_1 \hat v_h\|_{0,\hat K} \leq C[\hat f]_{k,\hat K} [\hat v_h]_{2,\hat K}.$$ 
So on a cell $e$, with \eqref{scaling}, we get
$$E(fv_h)=h^n\hat E (\hat f\hat v_h)=C h^{k+2}([ f]_{k+2, e} | v_h|_{0, e}+
[ f]_{k+1, e} | v_h|_{1, e}+[ f]_{k, e} [ v_h]_{2, e}).$$ 
 Summing over $e$ and use Cauchy Schwarz inequality, we get the desired result. 
\end{proof}
\begin{thm}
\label{rhs-inte-estimate}
 For $k\geq 2$, $(f,v_h)-(f_I,v_h) =\mathcal O(h^{k+2}) \|f\|_{k+2} \|v_h\|_2,\quad\forall v_h\in V^h.$
\end{thm}
\begin{proof}
Repeat the proof of Theorem \ref{rhs-estimate} for the function  $f-f_I$ on a cell $e$, with the fact $[f_I]_{k+1,p,e}=[f_I]_{k+2,p,e}=0$, we get 
$$E[(f-f_I)v_h]=C h^{k+2}([ f]_{k+2, e} | v_h|_{0, e}+
[ f]_{k+1, e} | v_h|_{1, e}+[ f-f_I]_{k, e} | v_h|_{2, e}).$$ 
By \eqref{bh1} on the Lagrange interpolation operator and the fact $[f-f_I]_{k,e}\leq \|f-f_I\|_{k+1,e}$, we get $[f-f_I]_{k,e}\leq Ch[f]_{k+1,e}$.
Notice that $\langle f-f_I, v_h\rangle_h=0$,  with \eqref{scaling}, we get
\[
(f,v_h)-(f_I,v_h)=(f-f_I,v_h)-\langle f-f_I, v_h\rangle_h=\mathcal O(h^{k+2}) \|f\|_{k+2} \|v_h\|_2,\forall v_h\in V^h.
\]

\end{proof}
\section{The M-type Projection}
\label{sec-projection}
To establish the superconvergence of $C^0$-$Q^k$  finite element method
for multi-dimensional variable coefficient equations, it is necessary to use
a special polynomial projection of the exact solution, which has two equivalent definitions. One is the 
M-type projection used in \cite{MR635547, chen2001structure}. 
The other one is 
the point-line-plane interpolation used in \cite{lin1991rectangle, lin1996}. 

For the sake of completeness, we review the relevant results regarding  M-type projection, which is a more convenient tool. Most results in this section were considered and established for more general rectangular elements  in \cite{chen2001structure}. 
For simplicity, we use some simplified proof and arguments for $Q^k$ element in this section. We only discuss the two dimensional case and the extension to three dimensions is straightforward. 
\subsection{One dimensional case}
The $L^2$-orthogonal  Legendre polynomials on the reference 
interval $\hat K=[-1,1]$ are given as 
\[l_k(t)=\frac{1}{2^k k!}\frac{d^k}{dt^k} (t^2-1)^k: l_0(t)=1, l_1(t)=t, l_2(t)=\frac12(3t^2-1), \cdots\]
Define their antiderivatives as M-type polynomials:
\[M_{k+1}(t)=\frac{1}{2^k k!}\frac{d^{k-1}}{dt^{k-1}} (t^2-1)^k: M_0(t)=1, M_1(t)=t, M_2(t)=\frac12(t^2-1), M_3(t)=\frac12(t^3-t),\cdots\]
which satisfy the following properties:
\begin{itemize}
 \item $M_k(\pm1)=0, \forall k\geq 2.$
 \item If $j-i\neq 0, \pm2$, then $M_i(t)\perp M_j(t)$, i.e., $\int_{-1}^1 M_i(t)M_j(t) dt=0.$
 \item Roots of $M_k(t)$ are the $k$-point Gauss-Lobatto quadrature points for $[-1,1]$. 
\end{itemize}
Since Legendre polynomials form a complete orthogonal basis for $L^{2}([-1,1])$,
for any $f(t)\in H^1([-1,1])$,  its derivative $f'(t)$ can be expressed as
Fourier-Legendre series
\[f'(t)=\sum_{j=0}^{\infty}b_{j+1}l_j(t), \quad b_{j+1}=(j+\frac12)\int_{-1}^1 f'(t)l_j(t)dt.\]
Define the M-type projection 
\[
f_k(t)=\sum_{j=0}^{k}b_{j}M_j(t),
\] 
where $b_0=\frac{f(1)+f(-1)}{2}$ is determined by $b_1=\frac{f(1)-f(-1)}{2}$
to make $f_k(\pm 1)=f(\pm 1)$. Since the Fourier-Legendre series converges in $L^2$, by Cauchy Schwarz inequality,
\[
\lim_{k\to \infty}f_k(t) -f(t)=\lim_{k\to \infty}\int_{-1}^t \left[f_k'(x)-f'(x)\right] dx \leq   \lim_{k\to \infty}\sqrt{2}\|f_k'(t)-f'(t)\|_{L^2([-1,1])}=0.
\]
We get the M-type expansion of $f(t)$:
$
f(t)=\lim\limits_{k\to \infty}f_k(t)=\sum\limits_{j=0}^{\infty}b_{j}M_j(t).
$
The remainder $R_k(t)$ of M-type projection  is
\[R[f]_k(t)=f(t)-f_k(t)=\sum_{j=k+1}^{\infty}b_{j}M_j(t).\]
The following properties are straightforward to verify: 
\begin{itemize}
 \item $f_k(\pm1)=f(\pm1)$ thus $R_k(\pm1)=0$ for $k\geq 1$. 
\item $R[f]_k(t)\perp v(t)$ for any $v(t)\in P^{k-2}(t)$ on $[-1,1]$, i.e., 
$\int_{-1}^1 R[f]_k v dt=0$. 
\item $R[f]_k'(t)\perp v(t)$ for any $v(t)\in P^{k-1}(t)$ on $[-1,1]$. 
\item For $j\geq 2$, $b_j=(j-\frac12)[\left.f(t)l_{j-1}(t)\right|_{-1}^1]
-\int_{-1}^1 f(t)l'(j-1)(t)dt.$
\item For $j\leq k$, $|b_j|\leq C_k \|f\|_{0, \infty, \hat K}.$
\item $\|R[f]_k(t)\|_{0, \infty, \hat K}\leq C_k\|f\|_{0, \infty, \hat K}.$

\end{itemize}
\subsection{Two dimensional case}
Consider a function $\hat f(s,t)\in H^2(\hat K)$ on the reference cell  $\hat K=[-1,1]\times[-1,1]$, it has the expansion  
\[\hat f(s,t)=\sum_{i=0}^\infty\sum_{j=0}^\infty \hat b_{i,j} M_i(s)M_j(t),\]
where 
\begin{align*}
\hat b_{0,0}&=\frac14[\hat f(-1,-1)+\hat f(-1,1)+\hat f(1,-1)+\hat f(1,1)],\\
\hat b_{0,j}, \hat b_{1,j}&=\frac{2j-1}{4}\int_{-1}^1 [\hat f_t(1,t)\pm \hat f_t(-1,t)]l_{j-1}(t)dt, \quad j\geq 1,\\
\hat b_{i,0}, \hat b_{i,1}&=\frac{2i-1}{4}\int_{-1}^1 [\hat f_s(s,1)\pm \hat f_s(s,-1)]l_{i-1}(s)ds, \quad i\geq 1,\\
\hat b_{i,j}&=\frac{(2i-1)(2j-1)}{4}\iint_{\hat K}\hat f_{st}(s,t)l_{i-1}(s)l_{j-1}(t)dsdt, 
\quad i,j\geq 1.\end{align*}
Define the $Q^k$ M-type projection of $\hat f$ on $\hat K$ and its remainder as
\[\hat f_{k,k}(s,t)=\sum_{i=0}^k\sum_{j=0}^k \hat b_{i,j} M_i(s)M_j(t), \quad \hat R[\hat f]_{k,k}(s,t)=\hat f(s,t)-\hat f_{k,k}(s,t).\]
For $f(x,y)$ on $e=[x_e-h, x_e+h]\times [y_e-h, y_e+h]$, let $\hat f(s,t)=
f( sh+x_e,  t h+y_e)$ then  the $Q^k$ M-type projection of $f$ on $e$ and its remainder are defined as 
\[f_{k,k}(x,y)=\hat f_{k,k}(\frac{x-x_e}{h},\frac{y-y_e}{h}), \quad R[ f]_{k,k}(x,y)= f(x,y)- f_{k,k}(x,y).\]

\begin{thm}
\label{plp-projection-theorem}
 The $Q^k$ M-type projection is equivalent to the $Q^k$ point-line-plane projection $\Pi$ defined as
 follows:
 \begin{enumerate}
  \item $\Pi \hat u=\hat u$ at four corners of $\hat K=[-1,1]\times[-1,1]$.
  \item $\Pi \hat u-\hat u$ is orthogonal to polynomials of degree $k-2$ on each edge of $\hat K$.
  \item $\Pi \hat u-\hat u$ is orthogonal to any $v\in Q^{k-2}(\hat K)$ on $\hat K$. 
  \end{enumerate}
\end{thm}
\begin{proof}
 We only need to show that  M-type projection $\hat f_{k,k}(s,t)$ satisfies 
 the same three properties. By $M_j(\pm1)=0$ for $j\geq 2$, we can derive that
 $\hat f_{k,k}=\hat f$ at $(\pm1,\pm1)$. For instance, $\hat f_{k,k}(1,1)=\hat b_{0,0}+\hat b_{1,0}+\hat b_{0,1}+\hat b_{1,1}=\hat f(1,1)$.
 
 The second property is implied by $M_j(\pm1)=0$ for $j\geq 2$ and $M_j(t)\perp P^{k-2}(t)$ for $j\geq k+1$. For instance, at $s=1$, $\hat f_{k,k}(1,t)-\hat f(1,t)=\sum\limits_{j=k+1}^\infty (\hat b_{0,j}+\hat b_{1,j}) M_j(t)\perp P^{k-2}(t)$ on $[-1,1]$. 
 
 The third property is implied by  $M_j(t)\perp P^{k-2}(t)$ for $j\geq k+1$.
 \end{proof}

\begin{lemma} 
\label{lemma-bij}
Assume $\hat f\in H^{k+1}(\hat K)$ with $k\geq 2$, then 
\begin{enumerate}
 \item $|\hat b_{i,j}|\leq C_k \|\hat f\|_{0,\infty, \hat K},\quad \forall i,j\leq k$.
 \item $|\hat b_{i,j}|\leq C_k |\hat f|_{i+j,2,\hat K},\quad \forall i,j\geq 1, i+j\leq k+1.$
 \item 
 $|\hat b_{i,k+1}|\leq C_k |\hat f|_{k+1,2,\hat K},\quad 0\leq i\leq k+1.$
  \item If $\hat f\in H^{k+2}(\hat K)$, then $|\hat b_{i,k+1}|\leq C_k |\hat f|_{k+2,2,\hat K},\quad 1\leq i\leq k+1.$
\end{enumerate}
\end{lemma}
\begin{proof}
 First of all, similar to the one-dimensional case, through integration by parts, $\hat b_{i,j}$ can be represented as integrals of $\hat f$ thus
 $|\hat b_{i,j}|\leq C_k \|\hat f\|_{0,\infty, \hat K}$ for $i,j\leq k$.
 
By the fact that the antiderivatives (and higher order ones) of Legendre polynomials vanish at $\pm1$, after integration by parts for both variables, we have
$$|\hat b_{i,j}|\leq C_k \iint_{\hat K} |\partial_s^i\partial_t^j\hat f| ds dt\leq C_k|\hat f|_{i+j,2,\hat K}, \quad \forall i,j\geq 1, i+j\leq k+1.$$

For the third estimate,  by integration by parts only for the variable $t$, we get 
$$\forall i\geq 1, |\hat b_{i,k+1}|\leq C_k \iint_{\hat K} |\partial_s\partial_t^{k}\hat f| ds dt\leq C_k|\hat f|_{k+1,2,\hat K}.$$
For $\hat b_{0,k+1}$, from the first estimate, we have $|\hat b_{0,k+1}|\leq C_k\|\hat f\|_{0,\infty,\hat K}\leq C_k \|\hat f\|_{k+1,2,\hat K}$ thus 
$\hat b_{0,k+1}$ can be regarded as a continuous linear form on $H^{k+1}(\hat K)$ and 
it vanishes if $\hat f \in Q^k(\hat K)$. So by the Bramble-Hilbert Lemma, 
$|\hat b_{0,k+1}|\leq C_k[\hat f]_{k+1,2,\hat K}$.

Finally, by integration by parts only for the variable $t$, we get 
$$|\hat b_{i,k+1}|\leq C_k \iint_{\hat K} |\partial_s\partial_t^{k+1}\hat f| ds dt\leq C_k|\hat f|_{k+2,2,\hat K},\quad 1\leq i\leq k+1.$$

\end{proof}

\begin{lemma}
\label{lemma-projection-remainder}
For $k\geq 2$, we have
 \begin{enumerate}
  \item 
  $|\hat R[\hat f]_{k,k}|_{0,\infty,\hat K}\leq C_k [\hat f]_{k+1,\hat K}$, $|\hat R[\hat f]_{k,k}|_{0,2, \hat K}\leq C_k [\hat f]_{k+1,\hat K}$.
\item $|\partial_s \hat R[\hat f]_{k,k}|_{0,\infty, \hat K}\leq C_k [\hat f]_{k+1,\hat K}$, $|\partial_s \hat R[\hat f]_{k,k}|_{0,2, \hat K}\leq C_k [\hat f]_{k+1,\hat K}$.
\item $\iint_{\hat K}\partial_s \hat R[\hat f]_{k,k} ds dt=0$
 \end{enumerate}
\end{lemma}
\begin{proof} 
 Lemma \ref{lemma-bij} implies $\|\hat f_{k,k}\|_{0,\infty, \hat K}\leq C_k \|\hat f\|_{0,\infty, \hat K}$
 and $\|\partial_s \hat  f_{k,k}\|_{0,\infty, \hat K}\leq C_k \|\hat f\|_{0,\infty, \hat K}$.
 Thus $$\forall (s,t)\in \hat K, |\hat R[\hat f]_{k,k}(s,t)|\leq |\hat f_{k,k}(s,t)|+|\hat f(s,t)|\leq C_k \|\hat f\|_{0,\infty, \hat K}\leq C_k \|\hat f\|_{k+1, \hat K}.$$
 Notice that here $C_k$ does not depend on $(s,t)$. 
So $R[\hat f]_{k,k}(s,t)$ is a continuous linear form on $H^{k+1}(\hat K)$ and its operator norm is bounded by a constant independent of $(s,t)$. 
Since it vanishes for any $\hat f\in Q^{k}(\hat K)$, by the Bramble-Hilbert Lemma, we get $|R[\hat f]_{k,k}(s,t)|\leq  C_k [\hat f]_{k+1,\hat K}$ where $C_k$ does not depend on $(s,t)$. So 
the $L^\infty$ estimate holds and it implies the $L^2$ estimate.

The second estimate can be established similarly since we have
$$ |\partial_s \hat R[\hat f]_{k,k}(s,t)|\leq |\partial_s \hat f_{k,k}(s,t)|+| \partial_s \hat f(s,t)|\leq C_k \|\hat f\|_{1,\infty, \hat K}\leq C_k \|\hat f\|_{k+1, \hat K}.$$

The third equation is implied by the fact that $M_j(t)\perp 1$ for $j\geq 3$ and $M'_j(t)\perp 1$ for $j\geq 2$. Another way to prove the third equation is to use integration by parts 
$$\iint_{\hat K}\partial_s \hat R[\hat f]_{k+1,k+1} ds dt=\int_{-1}^1 \left(\hat R[\hat f]_{k+1,k+1}(1,t)-\hat R[\hat f]_{k+1,k+1}(-1,t)\right)dt,$$ which is zero the second property in Theorem \ref{plp-projection-theorem}.
 \end{proof}

 For the discussion in the next few subsections, it is useful to consider the 
 lower order part of the remainder of $\hat R[\hat f]_{k,k}$:
 \begin{lemma}
 \label{lemma-remainder-highlow}
 For  $\hat f\in H^{k+2}(\hat K)$ with $k\geq 2$, 
define $\hat R[\hat f]_{k+1,k+1}-\hat R[\hat f]_{k,k}=\hat R_1+\hat R_2$ with
 \begin{align}
 \label{residue-splitting}
 \begin{split}
  \hat R_1&=\sum_{i=0}^k \hat b_{i,k+1}M_{i}(s)M_{k+1}(t),\\
  \hat R_2&=\sum_{j=0}^{k+1}\hat  b_{k+1,j}M_{k+1}(s)M_j(t) =M_{k+1}(s)\hat b_{k+1}(t),\quad 
 \hat b_{k+1}(t)= \sum_{j=0}^{k+1}\hat b_{k+1,j}M_j(t).
 \end{split}
 \end{align}
They have the following properties:
\begin{enumerate}
\item $\iint_{\hat K} \partial_s \hat R_1 ds dt=0$. 
 \item $|\partial_s \hat R_1|_{0,\infty,\hat K}\leq C_k|\hat f |_{k+2,2,\hat K}$, $|\partial_s \hat R_1|_{0,2,\hat K}\leq C_k|\hat f |_{k+2,2,\hat K}.$
 \item $|\hat b_{k+1}(t)|\leq C_k |\hat f|_{k+1,\hat K}$, 
 $|\hat b'_{k+1}(t)|\leq C_k |\hat f|_{k+2,\hat K}$, $\forall t\in [-1,1]$.
\end{enumerate}
 \end{lemma}
\begin{proof}
The first equation is due to the fact that $M_{k+1}(t)\perp 1$ since $k\geq 2.$

Notice that $M'_0(s)=0$, by Lemma \ref{lemma-bij}, we have
$$|\partial_s \hat R_1(s,t)|=\left|\sum_{i=1}^k \hat b_{i,k+1}M_{i}'(s)M_{k+1}(t)\right|\leq C_k|\hat f|_{k+2,\hat K}.$$ So we get the $L^\infty$ estimate for
$|\partial_s \hat R_1(s,t)|$ thus the $L^2$ estimate. 
 
 Similar to the estimates in Lemma \ref{lemma-bij}, we can show 
 $|\hat  b_{k+1,j}|\leq C_k |\hat f|_{k+1,\hat K}$ for $j\leq k+1$, thus $|b_{k+1}(t)|\leq C_k |\hat f|_{k+1,\hat K}$. 
  Since $b_{k+1}'(t)= \sum\limits_{j=1}^{k+1}\hat b_{k+1,j}M_j'(t)$, by the last 
 estimate in Lemma \ref{lemma-bij}, we get $|\hat b'_{k+1}(t)|\leq C_k |\hat f|_{k+2,\hat K}$.
\end{proof}

\subsection{The $C^0$-$Q^k$ projection}

Now consider a function $u(x,y)\in H^{k+2}(\Omega)$, let $u_p(x,y)$ denote its piecewise $Q^k$ M-type projection on each element $e$ in the mesh $\Omega_h$. 
The first two properties in Theorem \ref{plp-projection-theorem} imply that $u_p(x,y)$ on each edge is uniquely determined by $u(x,y)$ along that edge. Thus 
 $u_p(x,y)$ is continuous on $\Omega_h$.  The approximation error $u-u_p$ is one order higher at all Gauss-Lobatto points $Z_0$:
\begin{thm}
\label{thm-superapproximation}
 \[\|u-u_p\|_{2,Z_0}=\mathcal O(h^{k+2}) \|u\|_{k+2},\quad\forall u\in H^{k+2}(\Omega).\]
 \[\|u-u_p\|_{\infty,Z_0}=\mathcal O(h^{k+2}) \|u\|_{k+2,\infty},
 \quad\forall u\in W^{k+2,\infty}(\Omega).\]
\end{thm}
\begin{proof}
Consider any $e$ with cell center $(x_e, y_e)$, define $\hat u(s,t)=u(x_e+sh, y_e+th)$. 
 Since the $(k+1)$ Gauss-Lobatto points are roots of $M_{k+1}(t)$, 
 $\hat R_{k+1,k+1}[\hat u]-\hat R_{k,k}[\hat u]$ vanishes at $(k+1)\times(k+1)$ Gauss-Lobatto points on $\hat K$. 
 By Lemma \ref{lemma-projection-remainder}, we have $|\hat R_{k+1,k+1}[\hat u](s,t)|\leq C [\hat u]_{k+2,\hat K}$.
 
 Mapping back to the cell $e$, with \eqref{scaling}, at the $(k+1)\times(k+1)$ Gauss-Lobatto points on $e$, $|u-u_p|\leq C h^{k+2-\frac{n}{2}}[u]_{k+2, e}$. Summing over all elements $e$, we get
 \[\|u-u_p\|_{2, Z_0}\leq C \left [h^n\sum_e h^{2k+4-n}[u]_{k+2, e}^2\right ]^{\frac12}=\mathcal O (h^{k+2})[u]_{k+2,\Omega}. \]
 
 If further assuming $u\in W^{k+2,\infty}(\Omega)$, then 
 at the $(k+1)\times(k+1)$ Gauss-Lobatto points on $e$, $|u-u_p|\leq C h^{k+2-\frac{n}{2}}[u]_{k+2, e}\leq C h^{k+2} [u]_{k+2,\infty,\Omega}$, which implies 
 the second estimate. 
\end{proof}

\subsection{Superconvergence of bilinear forms}
\label{sec-bilinearform}

For convenience, in this subsection, we drop the subscript $h$ in a test function $v_h\in V^h$. When there is no confusion, we may also drop $dxdy$ or $dsdt$ in a double integral. 

\begin{lemma}
\label{lemma-bilinear-laplacian}
Assume $a(x,y)\in W^{2,\infty}(\Omega).$ For $k\geq 2$, 
 \[
  \iint_{\Omega}a(u-u_p)_xv_x \,dxdy=\mathcal O(h^{k+2})\|u\|_{k+2}\|v\|_2,\quad \forall v\in V^h.
\]
\end{lemma}
\begin{proof}
For each cell $e$, we consider $\iint_{e}a(u-u_p)_xv_x \,dxdy$. 
Let $R[u]_{k,k}=u-u_p$ denote the M-type projection remainder on $e$. 
Then $R[u]_{k,k}$ can be splitted into lower order part $R[u]_{k,k}-R[u]_{k+1,k+1}$ and high order part $R[u]_{k+1,k+1}$. 
\[\iint_{e} a(u-u_p)_xv_x \,dxdy=\iint_{e} a(R[u]_{k+1,k+1})_xv_x \,+\iint_{e} a(R[u]_{k,k}-R[u]_{k+1,k+1})_xv_x.\]
We first consider the high order part. Mapping everything to the reference cell $\hat K$ and let $\overline{\hat a\hat v_s}$ denote the average of $\hat a\hat v_s$ on $\hat K$.  
By the last property in Lemma \ref{lemma-projection-remainder}, we get
\begin{align*}
&h^{2-n}\left|\iint_{e} a(R[u]_{k+1,k+1})_xv_x \,dxdy\right|=\left|\iint_{\hat K} \partial_s (\hat R[\hat u]_{k+1,k+1}) \hat a\hat v_s ds dt\right|\\
&=
\left| \iint_{\hat K} \partial_s (\hat R[\hat u]_{k+1,k+1}) (\overline{\hat a\hat v_s}-{\hat a\hat v_s}) ds dt\right| 
\leq |\partial_s (\hat R[\hat u]_{k+1,k+1})|_{0,2,\hat K} |\overline{\hat a\hat v_s}-{\hat a\hat v_s}|_{0,2,\hat K}. 
\end{align*}
By Poincar\'{e} inequality and Cauchy-Schwarz inequality, we have
\[|\overline{\hat a\hat v_s}-{\hat a\hat v_s}|_{0,2,\hat K}\leq C |\nabla ({\hat a\hat v_s})|_{0,2,\hat K}\leq C|\hat a|_{1,\infty,\hat K} |\hat v|_{1,2,\hat K}+
C|\hat a|_{0,\infty,\hat K} |\hat v|_{2,2,\hat K}.\]
Mapping back to the cell $e$, with \eqref{scaling}, by Lemma \ref{lemma-projection-remainder}, the higher order part is bounded by
$C h^{k+2}[u]_{k+2,2,e}(|a|_{1,\infty, e} |v|_{1,2,e}+|a|_{0,\infty, e} |v|_{2,2,e})$ thus
\begin{align*}\sum_e \iint_{e} a(R[u]_{k+1,k+1})_xv_x \,dxdy&=\mathcal O(h^{k+2})\|a\|_{1,\infty,\Omega}\sum_e \|u\|_{k+2,e}\|v\|_{2,e}\\
&=\mathcal O(h^{k+2})\|a\|_{1,\infty,\Omega} \|u\|_{k+2,\Omega}\|v\|_{2,\Omega}.\end{align*}

Now we only need to discuss 
 the lower order part of the remainder. Let $R[u]_{k,k}-R[u]_{k+1,k+1}=R_1+R_2$ which is defined similarly as in \eqref{residue-splitting}. For $R_1$, by the first two results in Lemma \ref{lemma-remainder-highlow}, we have 
\begin{align*}
\iint_{\hat K}  (\partial_s\hat R_1) \hat a\hat v_s=\iint_{\hat K} (\partial_s\hat R_1) (\hat a\hat v_s-\overline{ \hat a\hat v_s})\leq 
|\partial_s \hat R_1|_{0,2,\hat K} |\overline{\hat a\hat v_s}-{\hat a\hat v_s}|_{0,2,\hat K} \\
\leq  C |\hat u|_{k+2,2,\hat K} |\overline{\hat a\hat v_s}-{\hat a\hat v_s}|_{0,2,\hat K}. 
\end{align*}
By similar discussions above, we get 
\begin{align*}
\sum_e \iint_{e} a(R_1)_xv_x \,dxdy=
\mathcal O(h^{k+2})\|a\|_{1,\infty,\Omega} \|u\|_{k+2,\Omega}\|v\|_{2,\Omega}.
\end{align*}

For $R_2$, let $N(s)$ be the antiderivative of $M_{k+1}(s)$ then 
$N(\pm1)=0$.
Let $\bar{\hat a}$ be the average of $\bar{\hat a}$ on $\hat K$ then 
$|\hat a-\bar{\hat a}|_{0,\infty, \hat K}\leq C |\hat a|_{1,\infty,\hat K}$.
Since $M_{k+1}(s)\perp P^{k-2}(s)$, we have $\iint_{\hat K}\hat  b_{k+1}(t) M_{k+1}(s)\hat v_{ss}=0.$
After integration by parts, by Lemma \ref{lemma-remainder-highlow} we have 
\begin{align*}
 &\iint_{\hat K} (\partial_s \hat R_2) \hat a\hat v_s=
-\iint_{\hat K} \hat b_{k+1}(t)M_{k+1}(s) (\hat a_s\hat v_s+\hat a \hat v_{ss})\\
=&\iint_{\hat K} \hat b_{k+1}(t)N(s) (\hat a_{ss}\hat v_s+\hat a_{s}\hat v_{ss})-\iint_{\hat K}\hat  b_{k+1}(t)M_{k+1}(s)(\hat a-\bar{\hat a})\hat v_{ss} \\
\leq& C|\hat u|_{k+1,\hat K}(|\hat a|_{2,\infty,\hat K}|\hat v|_{1,2,\hat K}+
|\hat a|_{1,\infty,\hat K}|\hat v|_{2,2,\hat K}).
\end{align*}

Thus we can get 
\[\sum_e \iint_{e} (\partial_x R_2) a\hat v_x dx dy =
\mathcal O(h^{k+2})\|a\|_{2,\infty,\Omega} \|u\|_{k+1,\Omega}\|v\|_{2,\Omega}.\]
So we have $\iint_{\Omega}a(u-u_p)_xv_x \,dxdy=\mathcal O (h^{k+2})\|a\|_{2,\infty,\Omega} \|u\|_{k+2}\|v\|_2,\quad \forall v\in V^h.$ 
\end{proof}

\begin{lemma}
Assume $c(x,y)\in W^{1,\infty}(\Omega).$ For $k\geq 2$, 
 \[
  \iint_{\Omega}c(u-u_p)v \,dxdy=\mathcal O(h^{k+2})\|u\|_{k+1}\|v\|_1,\quad \forall v\in V^h.
\]
\end{lemma}
\begin{proof}

Let $\overline{\hat c\hat v}$ be the average of $\hat c\hat v$ on $\hat K$. Following similar arguments as in the proof Lemma \ref{lemma-bilinear-laplacian},
\begin{align*}
\left|\iint_{\hat K}\hat  R[\hat u]_{k,k} \hat c \hat v\right|
=\left|\iint_{\hat K}\hat R[\hat u]_{k,k} (\hat c \hat v-\overline{\hat c\hat v})\right|
\leq|\hat R[\hat u]_{k,k}|_{0,2,\hat K}|\hat c \hat v-\overline{\hat c\hat v}|_{0,2,\hat K}\\
\leq C [u]_{k+1, 2,\hat K}  [\hat c \hat v]_{1,2,\hat K}
\leq C [u]_{k+1, 2,\hat K} (|\hat c|_{0,\infty, \hat K}| \hat v|_{1,2,\hat K}+
|\hat c|_{1,\infty, \hat K}| \hat v|_{0,2,\hat K}).  
\end{align*}
So with \eqref{scaling} we have
\[ \iint_{e} c R[ u]_{k,k} v dx dy=h^n  \iint_{\hat K} (R[\hat u]_{k,k}) \hat c\hat v ds dt=\mathcal O (h^{k+2})\|c\|_{1,\infty, \Omega}\|u\|_{k+1,e}\|v\|_{1,e},\]
which implies the estimate. 
\end{proof}

\begin{lemma}
Assume $b(x,y)\in W^{2,\infty}(\Omega).$ For $k\geq 2$, 
 \[
  \iint_{\Omega}b(u-u_p)_xv \,dxdy=\mathcal O(h^{k+2})\|u\|_{k+2}\|v\|_2,\quad \forall v\in V^h.
\]
\end{lemma}
\begin{proof}

Let $\overline{\hat b\hat v}$ be the average of $\hat b\hat v$ on $\hat K$. Following similar arguments as in the proof Lemma \ref{lemma-bilinear-laplacian}, we have
\begin{align*}&
\left|\iint_{\hat K} \partial_s (\hat R[\hat u]_{k+1,k+1}) \hat b\hat v\right|
=\left|\iint_{\hat K} \partial_s (\hat R[\hat u]_{k+1,k+1}) (\hat b\hat v-\overline{ \hat b\hat v})\right|\\
&\leq 
|\partial_s (\hat R[\hat u]_{k+1,k+1})|_{0,2,\hat K} |\overline{\hat b\hat v}-{\hat b\hat v}|_{0,2,\hat K}\leq 
C [\hat u]_{k+2,2,\hat K}(
|\hat b|_{1,\infty,\hat K} |\hat v|_{0,2,\hat K}+
|\hat b|_{0,\infty,\hat K} |\hat v|_{1,2,\hat K}). \end{align*}

\begin{align*}\iint_{\hat K} (\partial_s \hat R_1) \hat b\hat v=\iint_{\hat K} (\partial_s \hat R_1) (\hat b\hat v-\overline{ \hat b\hat v})\leq 
|\partial_s \hat R_1|_{0,2,\hat K} |\overline{\hat b\hat v}-{\hat b\hat v}|_{0,2,\hat K}\\
\leq C |\hat u|_{k+2,2,\hat K}(
|\hat b|_{1,\infty,\hat K} |\hat v|_{0,2,\hat K}+
|\hat b|_{0,\infty,\hat K} |\hat v|_{1,2,\hat K}).  \end{align*}

Let $N(s)$ be the antiderivative of $M_{k+1}(s)$. After integration by parts, we have 
\begin{align*}
 &\iint_{\hat K} (\partial_s \hat R_2) \hat b\hat v=
-\iint_{\hat K} \hat b_{k+1}(t)M_{k+1}(s) (\hat b_s\hat v+\hat b \hat v_{s})\\
=&\iint_{\hat K} \hat b_{k+1}(t)N(s) (\hat b_{ss}\hat v+\hat b_{s}\hat v_s+\hat b\hat v_{ss})\\
\leq& C|\hat u|_{k+1,2, \hat K}(|\hat b|_{2,\infty,\hat K}|\hat v|_{0,2,\hat K}+
|\hat b|_{1,\infty,\hat K}|\hat v|_{1,2,\hat K}+|\hat b|_{0,\infty,\hat K}|\hat v|_{2,2,\hat K}).
\end{align*}

After combining all the estimates, with \eqref{scaling}, we have 
\[  \iint_{e}b(u-u_p)_xv \,=h^{n-1} \iint_{\hat K}\hat b(R[\hat u]_{k,k})_s\hat v \,=\mathcal O(h^{k+2})\|b\|_{2,\infty,\Omega}\|u\|_{k+2,e}\|v\|_{2,e}.\]
\end{proof}

\begin{lemma}
\label{lemma-mixedderivative}
Assume $a(x,y)\in W^{2,\infty}(\Omega).$ For $k\geq 2$, 
 \begin{equation}
  \iint_{\Omega}a(u-u_p)_xv_y \,dxdy=\mathcal O(h^{k+2-\frac12})\|u\|_{k+2}\|v\|_2,\quad \forall v\in V^h,
\label{crossterm-1}
 \end{equation}
 \begin{equation}
 \iint_{\Omega}a(u-u_p)_xv_y \,dxdy=\mathcal O(h^{k+2})\|u\|_{k+2}\|v\|_2,\quad \forall v\in V^h_0.
\label{crossterm-2}
 \end{equation}
\end{lemma}
\begin{proof}
 Similar to the proof of Lemma \ref{lemma-bilinear-laplacian}, we have 
 \begin{align*}
&\left|\iint_{e} a(R[u]_{k+1,k+1})_xv_y \,dxdy\right|=h^{n-2}\left|\iint_{\hat K} \partial_s (\hat R[\hat u]_{k+1,k+1}) \hat a\hat v_t ds dt\right|\\
=&
h^{n-2}\left| \iint_{\hat K} \partial_s (\hat R[\hat u]_{k+1,k+1}) (\overline{\hat a\hat v_t}-{\hat a\hat v_t}) ds dt\right| 
\leq h^{n-2}|\partial_s (\hat R[\hat u]_{k+1,k+1})|_{0,2,\hat K} |\overline{\hat a\hat v_t}-{\hat a\hat v_t}|_{0,2,\hat K} \\
\leq& C h^{k+2} \|a\|_{1,\infty,\Omega}\|u\|_{k+2,e}\|v\|_{2,e},
\end{align*}
and
\[\iint_{\hat K} (\partial_s \hat R_1) \hat a\hat v_t=\iint_{\hat K} (\partial_s \hat R_1) (\hat a\hat v_t-\overline{ \hat a\hat v_t})\leq 
|\partial_s \hat R_1|_{0,2,\hat K} |\overline{\hat a\hat v_t}-{\hat a\hat v_t}|_{0,2,\hat K}.\]
Following the proof of Lemma \ref{lemma-bilinear-laplacian}, with \eqref{scaling}, we get 
\[\sum_e \iint_{e} a(R_1)_xv_y \,dxdy=\mathcal O(h^{k+2})\|a\|_{1,\infty,\Omega} \|u\|_{k+2,\Omega}\|v\|_{2,\Omega}.\]
Let $N(s)$ be the antiderivative of $M_{k+1}(s)$. After integration by parts, we have 
\begin{align*}
 &\iint_{\hat K} (\partial_s \hat R_2) \hat a\hat v_t=
-\iint_{\hat K} \hat b_{k+1}(t)M_{k+1}(s) (\hat a_s\hat v_t+\hat a \hat v_{st})\\
=&\iint_{\hat K} \hat b_{k+1}(t)N(s) (\hat a_{ss}\hat v_t+2\hat a_{s}\hat v_{st})+\iint_{\hat K} \hat b_{k+1}(t)N(s)\hat a\hat v_{sst}. 
\end{align*}

After integration by parts on the $t$-variable, 
\[-\iint_{\hat K} \hat b_{k+1}(t)N(s)\hat a\hat v_{sst} =\iint_{\hat K} \partial_t[\hat  b_{k+1}(t)N(s)\hat a]\hat v_{ss}-\left.\int_{-1}^1\hat  b_{k+1}(t)N(s)\hat a \hat v_{ss} ds\right|_{t=-1}^{t=1},\]
\[\iint_{\hat K} \partial_t[\hat b_{k+1}(t)N(s)\hat a]\hat v_{ss}=
\iint_{\hat K} [\hat b_{k+1}'(t)N(s)\hat a+\hat b_{k+1}(t)N(s)\hat a_t]\hat v_{ss}.\]

By Lemma \ref{lemma-remainder-highlow}, we have the estimate for the two double integral terms
\begin{align*}\left|\iint_{\hat K} \hat b_{k+1}(t)N(s) (\hat a_{ss}\hat v_t+2\hat a_{s}\hat v_{st})\right|
 \leq C|\hat u|_{k+1,2,\hat K}(|\hat a|_{2,\infty,\hat K}|\hat v|_{1,2,\hat K}+|\hat a|_{1,\infty,\hat K}|\hat v|_{2,2,\hat K}), 
\end{align*}
\begin{align*}
&\left|\iint_{\hat K} [\hat b_{k+1}'(t)N(s)\hat a+\hat b_{k+1}(t)N(s)\hat a_t]\hat v_{ss}\right|\\
\leq& C(|\hat u|_{k+2,2,\hat K}|\hat a|_{0,\infty,\hat K}|\hat v|_{2,2,\hat K}+
|\hat u|_{k+1,2,\hat K}|\hat a|_{1,\infty,\hat K}|\hat v|_{2,2,\hat K}), 
\end{align*}
which gives the estimate $C h^{k+2}\|a\|_{2,\infty,\Omega}\|u\|_{k+2,e}\|v\|_{k+2,e}$ after mapping back to $e$.

So we only need to discuss the line integral term now.
After mapping back to $e$, we have
\begin{align*}
 \left.\int_{-1}^1 \hat b_{k+1}(t)M_{k+1}(s)\hat a \hat v_{ss} ds\right|_{t=-1}^{t=1}= h\left. \int_{x_e-h}^{x_e+h} b_{k+1}(y)M_{k+1}(\frac{x-x_e}{h})a v_{xx} dx\right|_{y=y_e-h}^{y=y_e+h}. 
\end{align*}
Notice that we have
\begin{align*}b_{k+1}(y_e+h)=\hat b_{k+1}(1)=\sum_{j=0}^{k+1} \hat b_{k+1,j} M_j(1)=\hat b_{k+1,0}+ \hat b_{k+1,1}\\
=(k+\frac12)\int_{-1}^1 \partial_s \hat u(s,1) l_k(s)ds 
=(k+\frac12)\int_{x_e-h}^{x_e+h} \partial_x  u(x,y_e+h) l_k(\frac{x-x_e}{h})dx,\end{align*}
and similarly we get  $b_{k+1}(y_e-h)=\hat b_{k+1}(-1)=(k+\frac12)\int_{x_e-h}^{x_e+h} \partial_x  u(x,y_e-h) l_k(\frac{x-x_e}{h})dx$. Thus the term $b_{k+1}(y)M_{k+1}(\frac{x-x_e}{h})a v_{xx}$ is continuous across the top/bottom edge of cells. Therefore, if summing over all 
elements $e$, the line integral on the inner edges are cancelled out. Let $L_1$ and $L_3$ denote the top and bottom boundary of $\Omega$. Then the line integral after summing over $e$ consists of two line integrals along $L_1$ and $L_3$. We only need to discuss one of them. 

Let $l_1$ and $l_3$ denote the top and bottom edge of $e$. 
First, after integration by parts $k$ times, we get
\begin{align*}\hat b_{k+1}(1)
=(k+\frac12)\int_{-1}^1 \partial_s \hat u(s,1) l_k(s)ds=(-1)^k (k+\frac12)\int_{-1}^1 \frac{\partial^{k+1}}{\partial s^{k+1}} \hat u(s,1) \frac{1}{2^k k!}(s^2-1)^kds,\end{align*}
thus by Cauchy Schwarz inequality we get
\[|\hat b_{k+1}(1)|\leq C_k\sqrt{\int_{-1}^1 \left[\frac{\partial^{k+1}}{\partial s^{k+1}} \hat u(s,1)\right]^2  ds} \leq C_k h^{k+\frac12}|u|_{k+1,2,l_1}. \]
Second, since $v^2_{xx}$ is a polynomial of degree $2k$ w.r.t. $y$ variable, by using $(k+2)$-point Gauss Lobatto quadrature for integration w.r.t. $y$-variable in $\iint_e v^2_{xx} dx dy $, 
we get 
\[\int_{x_e-h}^{x_e+h} v^2_{xx}(x,y_e+h) dx \leq C h^{-1}\iint_{e} v^2_{xx}(x,y) dxdy. \]
So by Cauchy Schwarz inequality, we have 
$$\int_{x_e-h}^{x_e+h} |v_{xx}(x,y_e+h)| dx\leq \sqrt{2h} \sqrt{\int_{x_e-h}^{x_e+h} v^2_{xx}(x,y_e+h) dx}\leq C |v|_{2,2,e}.$$

Thus the line integral along $L_1$ can be estimated by considering each $e$ adjacent to $L_1$ in the reference cell: 
\begin{align*}&\sum_{e\cap L_1\neq \emptyset} \left|\int_{-1}^1 \hat b_{k+1}(1)M_{k+1}(s)\hat a(s,1) \hat v_{ss}(s,1) ds\right|
\\  \leq& \sum_{e\cap L_1\neq \emptyset}C |\hat a|_{0,\infty,\hat K} |\hat b_{k+1}(1)|\int_{-1}^1 |\hat v_{ss}(s,1)| ds\\
= &\mathcal O(h^{k+\frac32})\sum_{e\cap L_1\neq \emptyset}|u|_{k+1, 2, l_1} 
\int_{x_e-h}^{x_e+h} |v_{xx}(x,y_e+h)| dx\\
=& \mathcal O(h^{k+\frac32})\sum_{e\cap L_1\neq \emptyset}|u|_{k+1, 2, l_1} 
|v|_{2, 2, e}\\
=&\mathcal O(h^{k+\frac32}) \|u\|_{k+1, L_1} 
\|v\|_{2, \Omega}=\mathcal O(h^{k+\frac32}) \|u\|_{k+2, \Omega} 
\|v\|_{2, \Omega},
\end{align*}
where the trace inequality $ \|u\|_{k+1,\partial \Omega} \leq C  \|u\|_{k+2, \Omega}$ is used.

Combine all the estimates above, we get \eqref{crossterm-1}. Since the $\frac12$ order loss is only due to the line integral along $L_1$ and $L_3$, on which  $v_{xx}=0$ if $v\in V^h_0$, we get \eqref{crossterm-2}.
\end{proof}

\section{The main result}
\label{sec-main}
\subsection{Superconvergence of bilinear forms with  approximated coefficients}
\label{sec-main-pdesection}
Even though standard interpolation error is $a-a_I=\mathcal O (h^{k+1})$, as shown in the following discussion, the error in the bilinear forms is related to $\iint_{e} (a- a_I)\, dxdy$ on each cell $e$, which is the quadrature error thus the order is higher. We have the following estimate on the bilinear forms with approximated coefficients:
\begin{lemma}
\label{thm-term1}
Assume $a(x,y)\in W^{k+2,\infty}(\Omega)$ and $u(x,y)\in H^2(\Omega)$,  then $\forall v\in V^h$ or $\forall v\in H^2(\Omega),$
 \begin{align*}
  \iint_{\Omega}a u_x v_x\,dxdy-\iint_{\Omega}a_Iu_x v_x \,dxdy=\mathcal O(h^{k+2})\|a\|_{k+2,\infty,\Omega} \|u\|_2\|v\|_2, \\
  \iint_{\Omega}a u_x v_y\,dxdy-\iint_{\Omega}a_Iu_x v_y \,dxdy=\mathcal O(h^{k+2})\|a\|_{k+2,\infty,\Omega} \|u\|_2\|v\|_2, \\
  \iint_{\Omega}a u_x v \,dxdy-\iint_{\Omega}a_I u_x v\,dxdy=\mathcal O(h^{k+2})\|a\|_{k+2,\infty,\Omega} \|u\|_2\|v\|_1, \\
    \iint_{\Omega}a u v \,dxdy-\iint_{\Omega}a_I u v\,dxdy=\mathcal O(h^{k+2})\|a\|_{k+2,\infty,\Omega} \|u\|_1\|v\|_1. 
 \end{align*}
 \end{lemma}
\begin{proof}
  For every cell $e$ in the mesh $\Omega_h$, let $\overline{u_xv_x}$ be the cell average of ${u_xv_x}$. By Theorem  \ref{interp-theorem} and
  Theorem \ref{quaderror-theorem} , we have
 \begin{align*}
  &\iint_e (a_I-a)u_xv_x\\
  =&\iint_e (a_I-a)\overline{u_xv_x}+\iint_e (a_I-a)(u_xv_x-\overline{u_xv_x})\\
  =&\frac{1}{4h^2}\iint_e (a_I-a)\iint_e u_xv_x+\iint_e (a_I-a)(u_xv_x-\overline{u_xv_x})\\
  =& \mathcal O(h^{k+2})\|a\|_{k+2,\infty,\Omega}\|u\|_{1,e}\|v\|_{1,e}+\mathcal O(h^{k+1})\|a\|_{k+1,\infty,\Omega}\iint_e |u_xv_x-\overline{u_xv_x}|.
 \end{align*}
 
  By Poincar\'{e} inequality and Cauchy-Schwarz inequality, we have 
  $$\iint_e |u_xv_x-\overline{u_xv_x}|=\mathcal O(h)\|\nabla(u_xv_x)\|_{0,1,e}=\mathcal O(h)\|u\|_{2,e}\|v\|_{2,e}$$
  thus 
  $\iint_e (a_I-a)u_xv_x=\mathcal O(h^{k+2})\|a\|_{k+2,\infty,\Omega}\|u\|_{2,e}\|v\|_{2,e}.$
 Summing over all elements $e$, we have $\iint_{\Omega} (a_I-a)u_xv_x=\mathcal O(h^{k+2}) \|a\|_{k+2,\infty,\Omega} \|u\|_2\|v\|_2.$ Similarly we can establish the other three estimates. 
\end{proof}
Lemma 
\ref{thm-term1} implies that the difference in the solutions to  \eqref{Poisson-approximated} and \eqref{poisson} is  $\mathcal O(h^{k+2})$ in the $L^2(\Omega)$-norm:
\begin{thm} 
\label{thm-pde-error}
Assume $a(x,y)\in W^{k+2,\infty}(\Omega)$ and $a_I(x,y)\geq C>0$.
Let $u, \tilde u\in H_0^1(\Omega)$ be the solutions to 
$$A(u,v):=\iint a\nabla u\cdot\nabla v\,dxdy=(f,v),\quad\forall v \in H_0^1(\Omega)$$
and
$$A_I(\tilde u,v):=\iint a_I\nabla \tilde u\cdot\nabla v\,dxdy=(f,v),\quad\forall v \in H_0^1(\Omega)$$
respectively, where $f\in L^2(\Omega)$. Then $\|u-\tilde u\|_0=\mathcal O(h^{k+2}) \|a\|_{k+2,\infty,\Omega} \|f\|_0.$
\end{thm}
\begin{proof}
By
Lemma 
\ref{thm-term1}, for any $v\in H^2 (\Omega)$ we have
\begin{align*}
 A_I(u-\tilde u, v)=A_I(u, v)-A_I(\tilde u, v)=[A_I(u, v)-A(u, v)]+[A(u, v)-A_I(\tilde u, v)]\\
 =A_I(u, v)-A(u, v)=\mathcal O(h^{k+2})\|a\|_{k+2,\infty,\Omega} \| u\|_2 \|v\|_2.
\end{align*}
Let $w\in H_0^1(\Omega)$ be the solution to the dual problem
\begin{align*}
 A_I(v,w)=(u-\tilde u,v) \quad \forall v\in H^1_0(\Omega).
\end{align*}
Since $a_I\geq C>0$ and $|a_I(x,y)|\leq C|a(x,y)|$, the coercivity and boundedness of the bilinear form $A_I$ hold \cite{ciarlet1991basic}. Moreover, $a_I$ is Lipschitz continuous because $a(x,y)\in W^{k+2,\infty}(\Omega)$. Thus the solution $w$ exists and the elliptic regularity
$\|w\|_2\leq C \|u-\tilde u\|_0$ holds on a convex domain, e.g., a rectangular domain $\Omega$, see \cite{grisvard2011elliptic}. Thus, 
 \[\|u-\tilde u\|_0^2=(u-\tilde u, u-\tilde u)=A_I(u-\tilde u,w)=\mathcal O(h^{k+2})\|a\|_{k+2,\infty,\Omega} \|u\|_2\|w\|_2.\]
 With elliptic regularity $\|w\|_2\leq C \|u-\tilde u\|_0$  and $\|u\|_2\leq C\|f\|_0$, we get 
 \[\|u-\tilde u\|_0=\mathcal O(h^{k+2})\|a\|_{k+2,\infty,\Omega} \|f\|_0.\]
\end{proof}

\begin{rmk} For even number $k\geq 4$, $(k+1)$-point Newton-Cotes quadrature rule has the same error order as the $(k+1)$-point Gauss-Lobatto quadrature rule.
Thus Theorem \ref{thm-pde-error} still holds if we redefine $a_I(x,y)$ as the $Q^k$ interpolant of $a(x,y)$ at the uniform $(k+1)\times(k+1)$ Newton-Cotes points in each cell
if $k\geq 4$ is even. 
\end{rmk}

\subsection{The variable coefficient Poisson equation}

Let $u(x,y)\in H^1_0(\Omega)$ be the exact solution to 
\begin{equation*}
 A(u,v):=\iint_{\Omega} a \nabla u\cdot \nabla v \, dx dy=(f,v),\quad \forall v\in H^1_0(\Omega).
\end{equation*}
Let $\tilde u_h\in V^h_0(\Omega)$ be the  solution to 
\begin{equation*}
 A_I(\tilde u_h,v_h): =\iint_{\Omega} a_I \nabla \tilde u_h \cdot \nabla v_h \, dx dy=\langle f, v_h \rangle_h,\quad \forall v_h\in V^h_0(\Omega).
\end{equation*}

\begin{thm}
\label{theorem1}
For $k\geq 2$, let $u_p$ be the piecewise $Q^k$ M-type projection of $u(x,y)$ on each cell $e$ in the mesh $\Omega_h$. Assume $a\in W^{k+2,\infty}(\Omega)$ and $u, f\in H^{k+2}(\Omega)$, 
 then  
 $$A_I(\tilde{u}_h-u_p, v_h)=\mathcal O(h^{k+2})(\|a\|_{k+2,\infty}  \|u\|_{k+2}+\|f\|_{k+2})\|v_h\|_2,\quad \forall v_h\in V^h_0.$$
\end{thm}
\begin{proof}
 For any $v_h\in V^h$, we have 
\begin{align*}
&A_I(\tilde{u}_h,v_h)-A_I(u_p,v_h) \\
=&(f,v_h)-A_I(u_p,v_h)+\langle f,v_h\rangle_h-(f,v_h)\\
 =&A(u,v_h)-A_I(u_p,v_h)+\langle f,v_h\rangle_h-(f,v_h)\\
 =&[A(u,v_h)-A_I(u,v_h)]+[A_I(u-u_p,v_h)-A(u-u_p,v_h)]+A(u-u_p,v_h)+\langle f,v_h\rangle_h-(f,v_h).
 \end{align*}

Lemma \ref{thm-term1} implies $A(u,v_h)-A_I(u,v_h)=\mathcal O(h^{k+2})\|a\|_{k+2,\infty}\|u\|_2\|v_h\|_2$.
Theorem \ref{rhs-estimate} gives $\langle f,v_h\rangle_h-(f,v_h)=\mathcal O(h^{k+2})\|f\|_{k+2}\|v_h\|_2$. By Lemma \ref{lemma-bilinear-laplacian}, $A(u-u_p,v_h)=\mathcal O(h^{k+2})\|a\|_{2,\infty}\|u\|_{k+2}\|v_h\|_2$. 

 For the second term $A_I(u-u_p,v_h)-A(u-u_p,v_h)=\iint_\Omega (a-a_I) \nabla(u-u_p)\nabla v_h$,  by Theorem \ref{interp-theorem} and Lemma \ref{lemma-projection-remainder}, we have 
 \begin{align*}
 \left|\iint_\Omega (a-a_I) (u-u_p)_x \partial_x v_h\right|&\leq |a-a_I|_{0,\infty,\Omega} \sum_e \iint_e |(u-u_p)_x \partial_x v_h|\\
 &\leq |a-a_I|_{0,\infty,\Omega}\sum_e |(u-u_p)_x|_{0,2,e}| v_h|_{1,2,e}\\
 &= \mathcal O (h^{2k+1})\|a\|_{k+1,\infty,\Omega}\sum_e \|u\|_{k+1,e}\| v_h\|_{1,e}\\
 &=\mathcal O (h^{2k+1})\|a\|_{k+1,\infty,\Omega}\|u\|_{k+1}\| v_h\|_{1}.
 \end{align*}
\end{proof}

\begin{thm}
\label{mainthm}
 Assume $a(x,y)\in W^{k+2,\infty}(\Omega)$ is positive and $u(x,y), f(x,y)\in H^{k+2}(\Omega)$.
  Assume the mesh is fine enough so that the piecewise $Q^k$ interpolant satisfies $a_I(x,y)\geq C>0$.
 Then  $\tilde u_h$ is a ($k+2$)-th order accurate approximation to $u$ in the discrete 2-norm over all the $(k+1)\times(k+1)$ Gauss-Lobatto points:
 \[\|\tilde u_h-u\|_{2, Z_0}= \mathcal O (h^{k+2})(\|a\|_{k+2,\infty}\|u\|_{k+2}+\|f\|_{k+2}).\] 
\end{thm}

\begin{proof}
 
Let $\theta_h=\tilde{u}_h-u_p$.
By the definition of $u_p$ and Theorem \ref{plp-projection-theorem}, it is straightforward to show $\theta_h=0$ on $\partial \Omega$. 
By the Aubin-Nitsche duality method, 
let $w\in H_0^1(\Omega)$ be the solution to the dual problem
\begin{align*}
 A_I(v,w)=(\theta_h,v) \quad \forall v\in H^1_0(\Omega).
\end{align*}
By the same discussion as in the proof of Theorem \ref{thm-pde-error},
the solution $w$ exists and the regularity
$\|w\|_2\leq C \|\theta_h\|_0$ holds.

Let $w_h$ be the finite element projection of $w$, i.e., $w_h\in V_0^h$ satisfies
\begin{align*}
 A_I(v_h,w_h)=(\theta_h,v_h) \quad \forall v_h\in V^h_0.
\end{align*}

Since $w_h\in V^h_0$, by Theorem \ref{theorem1}, we have 
\begin{equation} \|\theta_h\|_0^2=(\theta_h,\theta_h)=A_I(\theta_h, w_h)=\mathcal O(h^4)(\|a\|_{k+2,\infty} \|u\|_{k+2}+\|f\|_{k+2})\|w_h\|_2.
\label{dual-eqn-4}
\end{equation}
Let $w_I=\Pi_1 w$ be the piecewise $Q^1$ projection of $w$ on $\Omega_h$ as defined in \eqref{projection1}.
By the Bramble-Hilbert Lemma, we get $\|w-w_I\|_{2,e}\leq C [w]_{2,e}\leq C\|w\|_{2,e}$ thus 
\[\|w-w_I\|_2\leq  C\|w\|_2. 
\]
By the inverse estimate on the piecewise polynomial $w_h-w_I$,  we have
\begin{equation}
\|w_h\|_2
\leq \|w_h-w_I\|_2+\|w_I-w\|_2+\|w\|_2\leq Ch^{-1}\|w_h-w_I\|_1+C\|w\|_2.\label{dual-eqn-1}
\end{equation}
With coercivity, Galerkin orthogonality and Cauchy Schwarz inequality, we get  
\[C\|w_h-w_I\|_1^2\leq A_I(w_h-w_I,w_h-w_I)=A_I(w_h-w_I,w-w_I)\leq C  \|w-w_I\|_1\|w_h-w_I\|_1,\]
which implies
\begin{equation}
\|w_h-w_I\|_1\leq C \|w-w_I\|_1\leq C h \|w\|_2. 
\label{dual-eqn-2}
\end{equation}

With \eqref{dual-eqn-1}, \eqref{dual-eqn-2} and the elliptic regularity
$\|w\|_2\leq C \|\theta_h\|_0$, we get  
\begin{equation}\|w_h\|_2\leq C\|w\|_2\leq  C \|\theta_h\|_0.
\label{dual-eqn-3}
\end{equation}
By \eqref{dual-eqn-4} and \eqref{dual-eqn-3} we have 
\begin{align*}
  \|\theta_h\|_0^2\leq \mathcal O (h^{k+2})(\|a\|_{k+2,\infty} \|u\|_{k+2}+\|f\|_{k+2})\|\theta_h\|_0,
\end{align*}
i.e., 
\begin{equation*}
 \|\tilde{u}_h-u_p\|_0= \|\theta_h\|_0= \mathcal O (h^{k+2})(\|a\|_{k+2,\infty} \|u\|_{k+2}+\|f\|_{k+2}).
\end{equation*}
Finally, by the equivalency between the discrete 2-norm on $Z_0$ and the
$L^2(\Omega)$ norm in the space $V^h$, with Theorem \ref{thm-superapproximation}, we obtain 
\[\|\tilde{u}_h-u\|_{2,Z_0}= \mathcal O (h^{k+2})(\|a\|_{k+2,\infty} \|u\|_{k+2}+\|f\|_{k+2}).\]
\end{proof}

\begin{rmk}
To extend Theorem \ref{mainthm} to homogeneous Neumann boundary conditions
or mixed homogeneous Dirichlet and Neumann boundary conditions, dual problems with the same homogeneous boundary conditions as in primal problems should be used. Then all the estimates such as Theorem \ref{theorem1} hold
not only for $v\in V_0^h$ but also for any $v$ in $V^h$. 
\end{rmk}
\begin{rmk}
With Theorem \ref{rhs-inte-estimate}, all the results hold for the scheme \eqref{numvarpro3}.
\end{rmk}
\begin{rmk} It is straightforward to verify that all results hold in three dimensions. 
Notice that the in three dimensions the discrete 2-norm is $$\|u\|_{2,Z_0}=\left[h^3\sum\limits_{\mathbf x\in Z_0} |u(\mathbf x)|^2\right]^{\frac12}.$$
\end{rmk}

\begin{rmk}
For  discussing superconvergence of the scheme \eqref{numvarpro4}, 
 we have to consider the dual problem of the bilinear form $A$ instead and the exact Galerkin orthogonality in \eqref{numvarpro4} no longer holds. 
In order for the proof above holds, we need to show the Galerkin orthogonality in \eqref{numvarpro4} holds up to $\mathcal O(h^{k+2})\|v_h\|_2$ for a test function $v_h\in V_h$, which is very difficult to establish. 
This is the main difficulty to extend the proof of Theorem \ref{mainthm} to the Gauss Lobatto quadrature scheme \eqref{numvarpro4}, which will be analyzed in \cite{li2019fourth} by different techniques.
\end{rmk}

\subsection{General elliptic problems}
In this section, we discuss extensions to more general elliptic problems.
Consider an elliptic variational problem  of finding $u\in H_0^1(\Omega)$ to satisfy
$$A(u,v): =\iint_\Omega (\nabla v^T \mathbf a \nabla u +\mathbf b\nabla u v + c u v) \,dx dy =(f,v), \forall v\in H^1_0(\Omega),$$
where $\mathbf a(x,y)=\begin{pmatrix}
               a_{11} & a_{12}\\
               a_{21} & a_{22}
              \end{pmatrix}
$ is positive definite and $\mathbf b=[b_1 \quad b_2]$. 
Assume the coefficients $\mathbf a$, $\mathbf b$ and $c$ are smooth, and $A(u,v)$ satisfies coercivity $A(v,v)\geq C \|v\|_1$ 
and boundedness $|A(u,v)|\leq C \|u\|_1 \|v\|_1$ for any $u,v\in H^1_0(\Omega)$.

By the estimates in Section \ref{sec-bilinearform}, we first have the following 
estimate on the $Q^k$ M-type projection $u_p$:
\begin{lemma} 
\label{lemma-elliptic-up}Assume $a_{ij}(x,y), b_i(x,y)\in W^{2,\infty}(\Omega)$ and $b_i(x,y)\in W^{2,\infty}(\Omega)$, then 
 $$A(u-u_p, v_h)=\left\{ \begin{array}{ll}
\mathcal O(h^{k+2})\|u\|_{k+2}\|v_h\|_2, \quad \forall v_h\in V^h_0,\\
\mathcal O(h^{k+1.5})\|u\|_{k+2}\|v_h\|_2, \quad \forall v_h\in V^h.
\end{array} \right.$$
If $a_{12}=a_{21}\equiv0$, then 
\[A(u-u_p, v_h)=\mathcal O(h^{k+2})\|u\|_{k+2}\|v_h\|_2, \quad\forall v_h\in V^h.\]
\end{lemma}

Let $\mathbf a_I$, $b_I$ and $c_I$ denote the 
corresponding piecewise $Q^k$ Lagrange interpolation at Gauss-Lobatto points.
We are interested in the solution $\tilde u_h\in V^h_0$ to 
$$A_I(\tilde u_h,v_h): =\iint_\Omega (\nabla v_h^T \mathbf a_I \nabla \tilde u_h +\mathbf b_I\nabla \tilde u_h v_h + c_I \tilde u_h v_h) \,dx dy =\langle f,v_h\rangle_h, \forall  v_h\in V^h_0.$$

We need to assume that $A_I$ still satisfies coercivity $A_I(v,v)\geq C \|v\|_1$ 
and boundedness $|A_I(u,v)|\leq C \|u\|_1 \|v\|_1$ for any $u,v\in H^1_0(\Omega)$, so that
the solution $u\in H^1_0(\Omega)$ of the following problem exists and is unique:
$$A_I(u,v)=(f,v), \quad \forall v\in H^1_0(\Omega).$$
We also need the elliptic regularity to hold for the dual problem:
$$A_I(v,w)=(f,v), \quad \forall v\in H^1_0(\Omega).$$

For instance, 
if $\mathbf b\equiv 0$, it suffices to require that 
eigenvalues of
$\mathbf a_I+c_I\begin{pmatrix}
                 1 & 0\\0& 1
                \end{pmatrix}
$ has a uniform positive lower bound on $\Omega$, which is achievable on fine enough meshes if 
$\mathbf a+c\begin{pmatrix}
                 1 & 0\\0& 1
                \end{pmatrix}
$ are positive definite. This implies the coercivity of $A_I$. The boundedness of $A_I$ follows from the smoothness of coefficients. Since $\mathbf a_I$ and $c_I$ are Lipschitz continuous, the elliptic regularity for $A_I$ holds on a convex domain \cite{grisvard2011elliptic}. 

By Lemma \ref{thm-term1} and Lemma \ref{lemma-elliptic-up}, it is straightforward to extend Theorem \ref{theorem1} to the general elliptic case:
\begin{thm}
For $k\geq 2$, assume $a_{ij}, b_i, c\in W^{k+2,\infty}(\Omega)$ and $u, f\in H^{k+2}(\Omega)$, 
 then  
 $$A_I(\tilde{u}_h-u_p, v_h)=\left\{ \begin{array}{ll}
\mathcal O(h^{k+2})(\|u\|_{k+2}+\|f\|_{k+2})\|v_h\|_2,\quad \forall v_h\in V^h_0,\\
\mathcal O(h^{k+1.5})(\|u\|_{k+2}+\|f\|_{k+2})\|v_h\|_2,\quad \forall v_h\in V^h,
\end{array} \right.
.$$
And if $a_{12}=a_{21}\equiv0$, then 
\[A_I(\tilde{u}_h-u_p, v_h)=\mathcal O(h^{k+2})(\|u\|_{k+2}+\|f\|_{k+2})\|v_h\|_2,\quad \forall v_h\in V^h.\]
\end{thm}

With suitable assumptions, 
it is straightforward to extend the proof of Theorem \ref{mainthm} to the general case:
\begin{thm}
\label{genmainthm}
For $k\geq 2$, assume $a_{ij}, b_i, c\in W^{k+2,\infty}(\Omega)$ and $u, f\in H^{k+2}(\Omega)$,
  Assume the approximated bilinear form $A_I$ satisfies coercivity and boundedness and the elliptic  regularity still holds for the dual problem of $A_I$.
 Then  $\tilde u_h$ is a ($k+2$)-th order accurate approximation to $u$ in the discrete 2-norm over all the $(k+1)\times(k+1)$ Gauss-Lobatto points:
 \[\|\tilde u_h-u\|_{2, Z_0}= \mathcal O (h^{k+2})(\|u\|_{k+2}+\|f\|_{k+2}).\] 
\end{thm}
\begin{rmk}
\label{remark-loss}
With Neumann type boundary conditions, due to Lemma \ref{lemma-mixedderivative}, we can only prove $(k+1.5)$-th order accuracy 
 \[\|\tilde u_h-u\|_{2, Z_0}= \mathcal O (h^{k+1.5})(\|u\|_{k+2}+\|f\|_{k+2}),\]
 unless there are no mixed second order derivatives in the elliptic equation, i.e., $a_{12}=a_{21}\equiv0.$
 We emphasize that even for the full finite element scheme \eqref{numvarpro1}, only $(k+1.5)$-th order accuracy at all Lobatto points can be proven for  a general elliptic equation with Neumann type boundary conditions.
 \end{rmk}

\section{Numerical results}
\label{sec-test}
In this section we show some numerical tests of $C^0$-$Q^2$ finite element method on an uniform rectangular mesh and verify the order of accuracy at $Z_0$, i.e., all Gauss-Lobatto points. The following four schemes will be considered:
\begin{enumerate}
\item Full $Q^2$ finite element scheme  \eqref{numvarpro1} where integrals in the bilinear form are approximated by $5\times 5$ Gauss quadrature rule, which is exact for $Q^{9}$ polynomials thus exact for $A(u_h,v_h)$ if the variable coefficient is a $Q^5$ polynomial. 
\item The   Gauss Lobatto quadrature scheme \eqref{numvarpro4}: all integrals are approximated by $3\times 3$ Gauss Lobatto quadrature.  
 \item The schemes \eqref{numvarpro2} and  \eqref{numvarpro3}.
\end{enumerate}
The last three schemes are finite difference type since only grid point values of the coefficients are needed. 
In   \eqref{numvarpro2} and  \eqref{numvarpro3}, $A_I(u_h, v_h)$ can be exactly computed by  $4\times 4$ Gauss quadrature rule since coefficients are $Q^2$ polynomials. An alternative finite difference type implementation of  \eqref{numvarpro2} and  \eqref{numvarpro3} is to precompute integrals of Lagrange basis functions and their derivatives to form a sparse tensor, then multiply the tensor to the vector consisting of point values of the coefficient to form the stiffness matrix. 
With either implementation, computational cost to assemble stiffness matrices in schemes \eqref{numvarpro2} and  \eqref{numvarpro3} is higher than 
the stiffness matrix assembling in the simpler scheme  \eqref{numvarpro4} since the Lagrangian $Q^k$ basis  are delta functions at Gauss-Lobatto points.
\subsection{Accuracy}
We consider the following example with either purely Dirichlet or purely  Neumann  boundary conditions:  
\begin{equation*}
\nabla\cdot(a\nabla u)=f\quad \textrm{on } [0,1]\times [0,2]
\end{equation*}
where $a(x,y)=1+0.1x^3y^5+\cos(x^3y^2+1)$ and $u(x,y)=0.1(\sin(\pi x)+x^3)(\sin(\pi y)+y^3)+\cos(x^4+y^3)$. 
The nonhomogeneous boundary condition should be computed in a way consistent with the computation of integrals in the bilinear form. 
The errors at $Z_0$ are shown in Table \ref{dirichlet} and Table \ref{neumann}. We can see that the four schemes are all fourth order in the discrete 2-norm on $Z_0$. Even though we did not discuss the max norm error on $Z_0$ in this paper, we should expect a $|\ln h|$ factor in the order of $l^\infty$ error over $Z_0$ due to \eqref{super-exactcoef-max}, which was proven upon the discrete Green's function.

\begin{table}[htbp]
\caption{The errors of $C^0$-$Q^2$ for a Poisson equation with Dirichlet boundary conditions at Lobatto points.}
\label{dirichlet}
\centering
\begin{tabular}{|c|c c|c c|}
\cline{1-5}
& \multicolumn{4}{|c|} {FEM with Approximated Coefficients \eqref{numvarpro2}}\\
\cline{1-5}
\hline  Mesh &  $l^2$ error  &  order & $l^\infty$ error & order \\
\hline
$2\times 4$  & 2.22E-1 & - & 3.96E-1 & -\\
\cline{1-5}
 $4\times 8$  & 4.83E-2 & 2.20 & 1.51E-1 & 1.39 \\
\cline{1-5}
$8\times 16$   & 2.54E-3 & 4.25 & 1.16E-2 & 3.71 \\ 
 \cline{1-5}
 $16\times 32$  & 1.49E-4 & 4.09 & 7.52E-4 & 3.95\\
 \cline{1-5}
$32\times 64$  & 9.22E-6 & 4.01 & 5.14E-5 & 3.87\\
\hline
& \multicolumn{4}{|c|} {FEM using Gauss Lobatto Quadrature  \eqref{numvarpro4}}  \\
\cline{1-5}
\hline  Mesh &  $l^2$ error  &  order & $l^\infty$ error & order \\
\hline
$2\times 4$  & 2.24E-1 & - & 4.30E-1 & -\\
\cline{1-5}
 $4\times 8$   & 4.43E-2 & 2.34 & 1.37E-1 & 1.65 \\
\cline{1-5}
$8\times 16$   & 2.27E-3 & 4.29 & 8.61E-3 & 4.00 \\ 
 \cline{1-5}
 $16\times 32$  & 1.32E-4 & 4.11 & 4.87E-4 & 4.14\\
 \cline{1-5}
 $32\times 64$  & 8.13E-6 & 4.02 & 3.09E-5 & 3.97\\
\hline
\cline{1-5}
& \multicolumn{4}{|c|} {FEM with Approximated Coefficients \eqref{numvarpro3}}\\
\cline{1-5}
\hline  Mesh &  $l^2$ error  &  order & $l^\infty$ error & order \\
\hline
$2\times 4$  & 2.78E-1 & - & 6.31E-1 & -\\
\cline{1-5}
 $4\times 8$  & 2.76E-2 & 3.33 & 8.69E-2 & 2.86 \\
\cline{1-5}
$8\times 16$   & 1.28E-3 & 4.43 & 3.77E-3 & 4.53 \\ 
 \cline{1-5}
 $16\times 32$  & 8.96E-5 & 3.83 & 3.36E-4 & 3.49\\
 \cline{1-5}
 $32\times 64$  & 5.79E-6 & 3.95 & 2.41E-5 & 3.80\\
\hline
& \multicolumn{4}{|c|} {Full FEM Scheme}  \\
\cline{1-5}
\hline  Mesh &  $l^2$ error  &  order & $l^\infty$ error & order \\
\hline
$2\times 4$  & 1.48E-2 & - & 3.79E-2 & -\\
\cline{1-5}
 $4\times 8$   & 1.05E-2 & 0.50 & 3.76E-2 & 0.01 \\
\cline{1-5}
$8\times 16$   & 7.32E-4 & 3.84 & 4.04E-3 & 3.22 \\ 
 \cline{1-5}
 $16\times 32$ & 4.54E-5 & 4.01 & 2.83E-4 & 3.83\\
 \cline{1-5}
$32\times 64$  & 2.85E-6 & 3.99 & 1.75E-5 & 4.02\\
\hline
\end{tabular}
\end{table}

\begin{table}[htbp]
\label{neumann}
\centering
\caption{The errors of $C^0$-$Q^2$ for a Poisson equation with Neumann boundary conditions at Lobatto points.}
\begin{tabular}{|c|c c|c c|}
\cline{1-5}
& \multicolumn{4}{|c|} {FEM with Approximated Coefficients \eqref{numvarpro2}}\\
\cline{1-5}
\hline  Mesh &  $l^2$ error  &  order & $l^\infty$ error & order \\
\hline
$2\times 4$  & 3.44E0 & - & 5.39E0 & -\\
\cline{1-5}
$4\times 8$   & 1.83E-1 & 4.23 & 3.51E-1 & 3.93\\
\cline{1-5}
$8\times 16$   & 1.38E-2 & 3.73 & 3.43E-2 & 3.36 \\ 
\cline{1-5}
$16\times 32$  & 8.37E-4 & 4.04 & 2.21E-3 & 3.96\\
\cline{1-5}
$32\times 64$  & 5.13E-5 & 4.03 & 1.41E-4 & 3.96\\
\hline
& \multicolumn{4}{|c|} {FEM using Gauss Lobatto Quadrature  \eqref{numvarpro4}}  \\
\cline{1-5}
\hline  Mesh &  $l^2$ error  &  order & $l^\infty$ error & order \\
\hline
$2\times 4$  & 3.43E0 & - & 4.95E0 & -\\
\cline{1-5}
$4\times 8$   & 1.81E-1 & 4.25 & 3.11E-1 & 3.99 \\
\cline{1-5}
$8\times 16$   & 1.37E-2 & 3.72 & 2.81E-2 & 3.47 \\ 
\cline{1-5}
$16\times 32$  & 8.33E-4 & 4.04 & 1.76E-3 & 4.00\\
\cline{1-5}
$32\times 64$  & 5.11E-5 & 4.03 & 1.12E-4 & 3.97\\
\hline
\cline{1-5}
& \multicolumn{4}{|c|} {FEM with Approximated Coefficients \eqref{numvarpro3}}\\
\cline{1-5}
\hline  Mesh &  $l^2$ error  &  order & $l^\infty$ error & order \\
\hline
$2\times 4$  & 3.64E0 & - & 5.06E0 & -\\
\cline{1-5}
$4\times 8$   & 1.60E-1 & 4.51 & 2.54E-1 & 4.32 \\
\cline{1-5}
$8\times 16$   & 1.26E-2 & 3.67 & 2.39E-2 & 3.41 \\ 
\cline{1-5}
$16\times 32$  & 7.67E-4 & 4.03 & 1.67E-3 & 3.84\\
\cline{1-5}
$32\times 64$  & 4.71E-5 & 4.03 & 1.09E-4 & 3.94\\
\hline
& \multicolumn{4}{|c|} {Full FEM Scheme}  \\
\cline{1-5}
\hline  Mesh &  $l^2$ error  &  order & $l^\infty$ error & order \\
\hline
$2\times 4$  & 8.45E-2 & - & 2.13E-1 & -\\
\cline{1-5}
$4\times 8$   & 1.56E-2 & 2.43 & 5.66E-2 & 1.91 \\
\cline{1-5}
$8\times 16$   & 9.12E-4 & 4.10 & 5.14E-3 & 3.46 \\ 
\cline{1-5}
$16\times 32$  & 5.47E-5 & 4.06 & 3.24E-4 & 3.99\\
\cline{1-5}
$32\times 64$  & 3.37E-6 & 4.02 & 2.22E-5 & 3.87\\
\hline
\end{tabular}
\end{table}

\begin{table}[htbp]
\label{neumann_elliptic}
\centering
\caption{An elliptic equation with mixed second order derivatives and Neumann boundary conditions.}
\begin{tabular}{|c|c c|c c|}
\cline{1-5}
& \multicolumn{4}{|c|} {FEM with Approximated Coefficients \eqref{numvarpro2}}\\
\cline{1-5}
\hline  Mesh &  $l^2$ error  &  order & $l^\infty$ error & order \\
\hline
 $2\times 4$   & 1.92E0 & - & 3.47E0 & -\\
\cline{1-5}
 $4\times 8$   & 2.16E-1 & 3.15 & 6.05E-1 & 2.52 \\
\cline{1-5}
 $8\times 16$  & 1.45E-2 & 3.90 & 6.12E-2 & 3.30 \\ 
 \cline{1-5}
 $16\times 32$  & 9.08E-4 & 4.00 & 4.05E-3 & 3.92\\
 \cline{1-5}
 $32\times 64$   & 5.66E-5 & 4.00 & 2.76E-4 & 3.88\\
\hline
& \multicolumn{4}{|c|} {FEM using Gauss Lobatto Quadrature  \eqref{numvarpro4}}  \\
\cline{1-5}
\hline  Mesh &  $l^2$ error  &  order & $l^\infty$ error & order \\
\hline
 $2\times 4$ & 1.38E0 & - & 2.27E0 & -\\
\hline
$4\times 8$ & 1.46E-1 & 3.24 & 2.52E-1 & 3.17\\
\hline
$8\times 16$  & 7.49E-3 & 4.28 & 1.64E-2 & 3.94 \\
\hline
$16\times 32$ & 4.31E-4 & 4.12 & 1.02E-3 & 4.01 \\ 
 \hline
 $32\times 64$  & 2.61E-5 & 4.04 & 7.47E-5 & 3.78\\
\hline
\cline{1-5}
& \multicolumn{4}{|c|} {FEM with Approximated Coefficients \eqref{numvarpro3}}\\
\cline{1-5}
\hline  Mesh &  $l^2$ error  &  order & $l^\infty$ error & order \\
\hline
 $2\times 4$  & 1.89E0 & - & 2.84E0 & -\\
\cline{1-5}
 $4\times 8$  & 1.04E-1 & 4.18 & 1.45E-1 & 4.30 \\
\cline{1-5}
 $8\times 16$  & 5.62E-3 & 4.21 & 1.86E-2 & 2.96 \\ 
 \cline{1-5}
 $16\times 32$  & 3.24E-4 & 4.12 & 1.67E-3 & 3.48\\
 \cline{1-5}
 $32\times 64$  & 1.95E-5 & 4.05 & 1.32E-4 & 3.66\\
\hline
& \multicolumn{4}{|c|} {Full FEM Scheme}  \\
\cline{1-5}
\hline  Mesh &  $l^2$ error  &  order & $l^\infty$ error & order \\
\hline
  $2\times 4$   & 1.46E-1 & - & 4.31E-1 & -\\
\cline{1-5}
 $4\times 8$   & 1.64E-2 & 3.16 & 6.55E-2 & 2.71 \\
\cline{1-5}
 $8\times 16$  & 7.08E-4 & 4.53 & 3.42E-3 & 4.26 \\ 
 \cline{1-5}
 $16\times 32$  & 4.44E-5 & 4.06 & 4.84E-4 & 2.82\\
 \cline{1-5}
 $32\times 64$   & 2.95E-6 & 3.85 & 7.96E-5 & 2.60\\
\hline
\end{tabular}
\end{table}

\begin{table}[htbp]
\label{dirichlet_elliptic}
\centering
\caption{An elliptic equation with mixed second order derivatives and Dirichlet boundary conditions.}
\begin{tabular}{|c|c c|c c|}
\cline{1-5}
& \multicolumn{4}{|c|} {FEM with Approximated Coefficients \eqref{numvarpro2}}\\
\cline{1-5}
\hline  Mesh &  $l^2$ error  &  order & $l^\infty$ error & order \\
\hline
 $2\times 4$   & 2.64E-2 & - & 7.01E-2 & -\\
\cline{1-5}
 $4\times 8$   & 4.68E-3 & 2.50 & 1.92E-2 & 1.87 \\
\cline{1-5}
 $8\times 16$  & 4.78E-4 & 3.29 & 2.70E-3 & 2.83 \\ 
 \cline{1-5}
 $16\times 32$  & 3.69E-5 & 3.69 & 2.43E-4 & 3.47\\
 \cline{1-5}
 $32\times 64$   & 2.53E-6 & 3.87 & 1.82E-5 & 3.74\\
 \cline{1-5}
 $64\times 128$   & 1.65E-7 & 3.94 & 1.25E-6 & 3.87\\
\hline
& \multicolumn{4}{|c|} {FEM using Gauss Lobatto Quadrature  \eqref{numvarpro4}}  \\
\cline{1-5}
\hline  Mesh &  $l^2$ error  &  order & $l^\infty$ error & order \\
\hline
$2\times 4$ & 3.94E-2 & - & 7.15E-2 & -\\
\hline
 $4\times 8$ & 1.23E-2 & 1.67 & 3.28E-2 & 1.12\\
\hline
 $8\times 16$ & 1.46E-3 & 3.08 & 5.42E-3 & 2.60 \\
\hline
$16\times 32$ & 1.14E-4 & 3.68 & 3.96E-4 & 3.78 \\ 
 \hline
 $32\times 64$ & 7.75E-6 & 3.88 & 2.62E-5 & 3.92\\
\hline
\cline{1-5}
& \multicolumn{4}{|c|} {FEM with Approximated Coefficients \eqref{numvarpro3}}\\
\cline{1-5}
\hline  Mesh &  $l^2$ error  &  order & $l^\infty$ error & order \\
\hline
 $2\times 4$   & 4.08E-2 & - & 7.67E-2 & -\\
\cline{1-5}
 $4\times 8$   & 1.01E-2 & 2.02 & 3.39E-2 & 1.18 \\
\cline{1-5}
 $8\times 16$  & 5.22E-4 & 4.27 & 1.72E-3 & 4.30 \\ 
 \cline{1-5}
 $16\times 32$  & 3.14E-5 & 4.05 & 9.57E-5 & 4.17\\
 \cline{1-5}
 $32\times 64$   & 1.99E-6 & 3.98 & 5.71E-6 & 4.07\\
\hline
& \multicolumn{4}{|c|} {Full FEM Scheme}  \\
\cline{1-5}
\hline  Mesh &  $l^2$ error  &  order & $l^\infty$ error & order \\
\hline
  $2\times 4$   & 7.35E-2 & - & 1.99E-1 & -\\
\cline{1-5}
 $4\times 8$   & 5.94E-3 & 3.63 & 2.43E-2 & 3.03 \\
\cline{1-5}
 $8\times 16$  & 4.31E-4 & 3.79 & 2.01E-3 & 3.60\\ 
 \cline{1-5}
 $16\times 32$  & 2.83E-5 & 3.93 & 1.76E-4 & 3.93\\
 \cline{1-5}
 $32\times 64$   & 1.68E-6 & 4.07 & 8.41E-6 & 4.07\\
\hline
\end{tabular}
\end{table}

\begin{table}[htbp]
\label{robust}
\centering
\caption{A Poisson equation with coefficient  $\min\limits_{(x,y)} a(x,y)\approx 0.001$.}
\begin{tabular}{|c|c c|c c|}
\cline{1-5}
& \multicolumn{4}{|c|} {FEM with Approximated Coefficients \eqref{numvarpro2}}\\
\cline{1-5}
\hline  Mesh &  $l^2$ error  &  order & $l^\infty$ error & order \\
\hline
 $2\times 4$  & 2.78E-1 & - & 4.52E-1 & -\\
\cline{1-5}
 $4\times 8$  & 6.22E-2 & 2.16 & 2.08E-1 & 1.12 \\
\cline{1-5}
 $8\times 16$ & 1.09E-2 & 2.51 & 8.44E-2 & 1.30 \\ 
 \cline{1-5}
 $16\times 32$  & 1.31E-3 & 3.05 & 1.81E-2 & 2.22\\
 \cline{1-5}
$32\times 64$  & 1.08E-4 & 3.60 & 1.75E-3 & 3.38\\
\hline
$64\times 128$  & 7.24E-6 & 3.90 & 1.52E-4 & 3.53\\
\hline
& \multicolumn{4}{|c|} {FEM using Gauss Lobatto Quadrature  \eqref{numvarpro4}}  \\
\cline{1-5}
\hline  Mesh &  $l^2$ error  &  order & $l^\infty$ error & order \\
\hline
 $2\times 4$  & 2.81E-1 & - & 4.59E-1 & -\\
\cline{1-5}
 $4\times 8$   & 4.69E-2 & 2.58 & 1.37E-1 & 1.74\\
\cline{1-5}
 $8\times 16$   & 5.06E-3 & 3.21 & 3.75E-2 & 1.87\\ 
 \cline{1-5}
 $16\times 32$   & 7.04E-4 & 2.85 & 7.86E-3 & 2.25\\
 \cline{1-5}
$32\times 64$   & 6.74E-5 & 3.39 & 1.21E-3 & 2.70\\
\hline
$64\times 128$   & 4.94E-6 & 3.77 & 1.17E-4 & 3.37\\
\hline
& \multicolumn{4}{|c|} {FEM with Approximated Coefficients \eqref{numvarpro3}}\\
\cline{1-5}
\hline  Mesh &  $l^2$ error  &  order & $l^\infty$ error & order \\
\hline
 $2\times 4$  & 2.68E-1 & - & 5.48E-1 & -\\
\cline{1-5}
 $4\times 8$ & 2.91E-1 & 3.21 & 1.59E-1 & 1.78 \\
\cline{1-5}
 $8\times 16$  & 3.51E-3 & 3.05 & 4.02E-2 & 1.98 \\ 
\cline{1-5}
 $16\times 32$ & 2.86E-4 & 3.62 & 3.60E-3 & 3.48 \\
\cline{1-5}
$32\times 64$  & 1.86E-5 & 3.94 & 2.31E-4 & 3.96\\
\cline{1-5}
$64\times 128$  & 1.17E-6 & 4.00 & 1.53E-5 & 3.91\\
\hline
\end{tabular}
\end{table}

Next we consider an elliptic equation with purely Dirichlet or purely Neumann boundary conditions:
\begin{equation*}
\nabla\cdot(\mathbf a\nabla u)+c u=f\quad \textrm{on } [0,1]\times [0,2]
\end{equation*}
where $\mathbf a=\left( {\begin{array}{cc}
   a_{11} & a_{12} \\
   a_{21} & a_{22} \\
  \end{array} } \right)$, $a_{11}=10+30y^5+x\cos{y}+y$, $a_{12}=a_{21}=2+0.5(\sin(\pi x)+x^3)(\sin(\pi y)+y^3)+\cos(x^4+y^3)$, $a_{22}=10+x^5$, $c=1+x^4y^3$ and $u(x,y)=0.1(\sin(\pi x)+x^3)(\sin(\pi y)+y^3)+\cos(x^4+y^3)$.
The errors at $Z_0$ are listed in Table \ref{neumann_elliptic} and  Table \ref{dirichlet_elliptic}. Recall that only $\mathcal O (h^{3.5})$ can be proven due to the mixed second order derivatives for the Neumann boundary conditions as discussed in Remark \ref{remark-loss},
we observe around fourth order accuracy for \eqref{numvarpro2} and  \eqref{numvarpro3} for Neumann boundary conditions in this particular example. 

\subsection{Robustness}

In Table \ref{dirichlet} and Table \ref{neumann}, the errors of approximated coefficient schemes  \eqref{numvarpro2},  \eqref{numvarpro3} and the Gauss Lobatto quadrature scheme  \eqref{numvarpro4} are close to one another. 
We observe that the scheme \eqref{numvarpro3} tends to be more accurate  than  \eqref{numvarpro2} and  \eqref{numvarpro4} when the coefficient
$a(x,y)$ is closer to zero in the Poisson equation. 
See Table \ref{robust} for errors of solving $\nabla\cdot(a\nabla u)=f\quad \textrm{on } [0,1]\times [0,2]$ with Dirichlet boundary conditions, 
$a(x,y)=1+\varepsilon x^3y^5+\cos(x^3y^2+1)$ and $u(x,y)=0.1(\sin(\pi x)+x^3)(\sin(\pi y)+y^3)+\cos(x^4+y^3)$ where $\varepsilon=0.001$.
Here the smallest value of $a(x,y)$ is around $\varepsilon=0.001$. We remark that the difference among three schemes is much smaller for larger $\varepsilon$ such as $\varepsilon=0.1$ as in Table \ref{dirichlet}.

\section{Concluding remarks}
\label{sec-remark}

We have shown that the classical superconvergence of functions values at Gauss Lobatto points in $C^0$-$Q^k$ finite element method for an elliptic problem still holds if replacing the coefficients by their piecewise $Q^k$ Lagrange interpolants at the Gauss Lobatto points. 
Such a superconvergence result can be used for constructing a fourth order accurate finite difference type scheme by using $Q^2$ approximated variable coefficients. 
Numerical tests suggest that this is an efficient and robust implementation of $C^0$-$Q^2$ finite element method  without affecting the superconvergence of function values.

\section*{Acknowledgments} Research is supported by the NSF grant DMS-1522593.
The authors are grateful to Prof. Johnny Guzm\'{a}n for discussions on Theorem \ref{thm-pde-error}.

\clearpage

\bibliographystyle{siamplain}
\bibliography{references.bib}

\end{document}